\documentclass[10pt,twocolumn,twoside]{IEEEtran}

\usepackage{amsmath,amssymb,mathrsfs,bm}
\usepackage{amsthm}
\usepackage{caption,subfigure}
\usepackage{graphicx,color,xspace}
\usepackage{url}
\makeatletter
\g@addto@macro{\UrlBreaks}{\UrlOrds}
\makeatother
\usepackage{multicol,lipsum}
\usepackage{dsfont}
\usepackage{tikz}

\usepackage[ruled]{algorithm2e}
\makeatletter
\renewcommand{\@algocf@capt@plain}{above}
\makeatother

\newcommand{\oprocendsymbol}{\hbox{$\bullet$}}
\newcommand{\oprocend}{\relax\ifmmode\else\unskip\hfill\fi\oprocendsymbol}
\def\eqoprocend{\tag*{$\bullet$}}

\newcommand{\longthmtitle}[1]{\mbox{}\textup{\textbf{(#1):}}}

\newtheorem{theorem}{Theorem}[section]
\newtheorem{lemma}[theorem]{Lemma}

\newtheorem{remark}[theorem]{Remark}
	
\newtheorem{corollary}[theorem]{Corollary}
\newtheorem{proposition}[theorem]{Proposition}
\newtheorem{assumption}[theorem]{Assumption}

\newcommand{\phicc}{\bm{\phi}}

\newcommand{\lowergamma}{\underline{\gamma}}
\newcommand{\dualoptset}{\zset^*}
\newcommand{\zsetcom}{\zset_c}

\newcommand{\mbk}{\mb}
\newcommand{\tlapbk}{\tlapb}
\newcommand{\lapkbk}{\lapkb}

\newcommand{\subdarg}[1]{d_{#1}}

\newcommand{\dimwn}{d_{\wn}}

\newcommand{\dimzn}{d_{\zn}}
\newcommand{\dimmun}{d_{\mun}}
\newcommand{\dimdn}{d_{\dnm}}

\newcommand{\realdimzn}{(\real^{\dimzn})^N}

\newcommand{\realdimdn}{(\real^{\dimdn})^N}
\newcommand{\realdimw}{\real^{\dimwn}}
\newcommand{\realdimz}{\real^{\dimzn}}
\newcommand{\realdimmu}{\real^{\dimmun}}
\newcommand{\realdimd}{\real^{\dimdn}}

\newcommand{\fen}{^\star}

\newcommand{\lag}{\mathcal{L}}

\newcommand{\subgwnt}{g_{\bm{w}_t}}
\newcommand{\subgwns}{g_{\bm{w}_s}}

\newcommand{\subgznt}{g_{\zn_t}}

\newcommand{\projec}[2]{\mathcal{P}_{#1}\big(#2\big)}

\newcommand{\estboundzn}{B_{\zn}}
\newcommand{\estboundwn}{B_{\wn}}
\newcommand{\estboundmun}{B_{\mun}}
\newcommand{\estbounddnm}{B_{\dnm}}

\newcommand{\pertdnmt}{\bm{u}^1_t}

\newcommand{\pertznt}{\bm{u}^2_t}
\newcommand{\persdnmt}{\bm{u}^1_s}

\newcommand{\bsubgzn}{H_{\zn}}
\newcommand{\bsubgwn}{H_{\wn}}
\newcommand{\bsubgmun}{H_{\mun}}
\newcommand{\bsubgdnm}{H_{\dnm}}

\newcommand{\consissu}{C_{u}}

\newcommand{\mun}{\bm{\mu}}
\newcommand{\munt}{\mun_t}

\newcommand{\rnxtp}{\mathbf{r}_{\dnm,t+1}}

\newcommand{\rnwtp}{\mathbf{r}_{\wn,t+1}}

\newcommand{\wnsetan}{\boldsymbol{\mathcal{W}}}
\newcommand{\wset}{\mathcal{W}}

\newcommand{\wseti}{\wset_i}
\newcommand{\prodwseti}{(\wset_1\times\cdots\times\wset_N)}
\newcommand{\prodwsetiwo}{\wset_1\times\cdots\times\wset_N}

\newcommand{\zset}{\mathcal{Z}}

\newcommand{\znset}{\zset^N}

\newcommand{\munsetan}{\boldsymbol{\mathcal{M}}}
\newcommand{\cartnset}{\bm{S}}

\newcommand{\dnmp}{\dnm_{p}}
\newcommand{\dnm}{\text{\scalebox{0.8}{$\bm{D}$}}}

\newcommand{\dm}{\text{\scalebox{0.8}{$D$}}}

\newcommand{\subgmunt}{g_{\mun_t}}

\newcommand{\subgdnmt}{g_{\dnm_t}}
\newcommand{\subgdnms}{g_{\dnm_s}}

\newcommand{\dnmset}{\dmset^N}
\newcommand{\dmset}{\mathcal{D}}

\newcommand{\dnmt}{\dnm_t}
\newcommand{\dnmtp}{\dnm_{t+1}}

\newcommand{\dnmstar}{\dnm^*}

\newcommand{\dnmtphat}{\hat{\dnm}_{t+1}}

\newcommand{\dnms}{\dnm_s}

\newcommand{\wntphat}{\hat{\wn}_{t+1}}

\newcommand{\gradsteps}{\eta_s}
\newcommand{\gradstepsm}{\eta_{s-1}}

\newcommand{\wnp}{\wn_{p}}
\newcommand{\wns}{\wn_{s}}

\newcommand{\zns}{\zn_{s}}
\newcommand{\znp}{\zn_{p}}

\newcommand{\muns}{\mun_{s}}
\newcommand{\munp}{\mun_{p}}

\newcommand{\dnmtav}{\dnm_t^{\text{av}}}
\newcommand{\zntav}{\zn_t^{\text{av}}}
\newcommand{\muntav}{\mun_t^{\text{av}}}
\newcommand{\wntav}{\wn_t^{\text{av}}}

\newcommand{\dmit}{\dm_t^i}

\newcommand{\wit}{w_t^i}

\newcommand{\dmiT}{\dm_T^i}
\newcommand{\ziT}{z_T^i}

\newcommand{\wiT}{w_T^i}

\newcommand{\dmitp}{\dm_{t+1}^i}
\newcommand{\zitp}{z_{t+1}^i}

\newcommand{\witp}{w_{t+1}^i}

\newcommand{\dmitpav}{{(\dmitp)}^{\text{av}}}
\newcommand{\zitpav}{{(\zitp)}^{\text{av}}}
\newcommand{\witpav}{{(\witp)}^{\text{av}}}

\newcommand{\dmitav}{{(\dmit)}^{\text{av}}}
\newcommand{\zitav}{{(\zit)}^{\text{av}}}
\newcommand{\witav}{{(\wit)}^{\text{av}}}

\newcommand{\dmiTav}{{(\dmiT)}^{\text{av}}}
\newcommand{\ziTav}{{(\ziT)}^{\text{av}}}
\newcommand{\wiTav}{{(\wiT)}^{\text{av}}}

\newcommand{\munstar}{\mun^*}

\newcommand{\muntpav}{\mun_{t+1}^{\text{av}}}

\newcommand{\dnmtpav}{\dnm_{t+1}^{\text{av}}}
\newcommand{\dnmtspav}{\dnm_{{t}'+1}^{\text{av}}}

\newcommand{\zntpav}{\zn_{t+1}^{\text{av}}}

\newcommand{\wntpav}{\wn_{t+1}^{\text{av}}}
\newcommand{\wntspav}{\wn_{{t}'+1}^{\text{av}}}

\newcommand{\zbar}{\bar{z}}

\newcommand{\sign}{\operatorname{sign}}

\newcommand{\diam}{\operatorname{diam}}

\newcommand{\ballc}[2]{\bar{\mathcal{B}}(#1,#2)}
\newcommand{\ballcbig}[2]{\bar{\mathcal{B}}\Big(#1,#2\Big)}

\newcommand{\degnt}{{\tilde{\delta}}}
\newcommand{\degn}{\delta}

\newcommand{\doutmaxT}{\subscr{d}{\text{max}}}

\newcommand{\doutt}{\subscr{d}{out,t}}

\newcommand{\inputet}{e_t}
\newcommand{\inputes}{e_s}

\newcommand{\upperb}{\mathsf{u}}

\newcommand{\wn}{\bm{w}}
\newcommand{\wnt}{\bm{w}_t}
\newcommand{\wntp}{\bm{w}_{t+1}}

\newcommand{\wnstarwo}{\bm{w}^*}

\newcommand{\lambdaup}{\overline{\Lambda}}

\newcommand{\znstarwo}{\zn^*}

\newcommand{\zn}{\bm{z}}

\newcommand{\zntphat}{\hat{\zn}_{t+1}}

\newcommand{\muntphat}{\hat{\mun}_{t+1}}

\newcommand{\znt}{\zn_{t}}
\newcommand{\zntp}{\zn_{t+1}}

\newcommand{\muntp}{\mun_{t+1}}

\newcommand{\zit}{{z}_{t}^i}

\newcommand{\tlapb}{\mathbf{L}_t}

\newcommand{\tlap}{\mathsf{L}_t}

\newcommand{\lapkbdouble}{\hat{\mathbf{L}}_{\mathcal{K}}}

\newcommand{\cstep}{\sigma}
\newcommand{\gradstept}{\eta_t}

\newcommand{\gradstep}{\eta}

\newcommand{\step}{\sigma}

\newcommand{\norm}[1]{\|#1\|_{\text{\scalebox{1}{$2$}}}}
\newcommand{\normone}[1]{\|#1\|_{\text{\scalebox{1}{$1$}}}}

\newcommand{\cdoubling}[2]{\hat{\alpha}_{#1, #2}}
\newcommand{\cdoublingfac}[2]{\alpha_{#1, #2}}

\newcommand{\absolute}[1]{|#1|}
\newcommand{\diag}{\operatorname{diag}}

\newcommand{\lambdamax}{\operatorname{\lambda_{\text{max}}}}

\newcommand{\lambdamin}{\operatorname{\lambda_{\text{min}}}}
\newcommand{\sigmamax}{\operatorname{\sigma_{\text{max}}}}

\newcommand{\tp}{^{\text{\scalebox{0.8}{$\top$}}}}
\newcommand{\defin}{:=}

\newcommand{\identity}{\mathrm{I}}

\newcommand{\zeros}[1]{\mathrm{0}_{#1}}

\newcommand{\ones}{\mathds{1}}

\newcommand{\graph}{\mathcal{G}}
\newcommand{\edgeset}{\mathcal{E}}
\newcommand{\vertexset}{\mathcal{I}}

\newcommand{\din}{\subscr{d}{in}}

\newcommand{\dout}{\subscr{d}{out}}

\newcommand{\Nout}{\upscr{\mathcal{N}}{out}}

\newcommand{\Adj}{\mathsf{A}}
\newcommand{\adj}{\mathsf{a}}

\newcommand{\lapkb}{\mathbf{L}_{\mathcal{K}}}

\newcommand{\lap}{\mathsf{L}}

\newcommand{\lapk}{\mathsf{L}_{\mathcal{K}}}

\newcommand{\matrixM}{\mathrm{M}}

\newcommand{\mb}{\mathbf{M}}

\newcommand{\realmatricesarg}[1]{\mathbb{R}^{#1\times #1}}

\newcommand{\realnonnegmatricesarg}[1]{\mathbb{R}_{\text{\scalebox{0.7}{$\ge 0$}}}^{#1\times #1}}

\newcommand{\real}{{\mathbb{R}}}

\newcommand{\realpositive}{{\mathbb{R}}_{> 0}}
\newcommand{\realnonnegative}{{\mathbb{R}}_{\ge 0}}

\newcommand{\integerspositive}{\mathbb{Z}_{\geq 1}}

\newcommand{\eps}{\epsilon}
\newcommand{\argmin}{\operatorname{arg \min}}
\newcommand{\until}[1]{\{1,\dots,#1\}}
\newcommand{\map}[3]{#1:#2 \rightarrow #3}

\newcommand{\setdef}[2]{\{#1 \, : \, #2\}}
\newcommand{\setdefbig}[2]{\big\{\,#1 \; : \; #2\,\big\}}

\newcommand\subscr[2]{#1_{\textup{#2}}}
\newcommand\upscr[2]{#1^{\textup{#2}}}

\begin{document}

\title{Distributed saddle-point subgradient algorithms\\ with
  Laplacian averaging\thanks{A preliminary version of this work has
    been accepted as~\cite{DMN-JC:15-cdc} at the 2015 IEEE Conference
    on Decision and Control, Osaka, Japan}}

\author{David Mateos-N\'u\~nez \qquad Jorge Cort\'es \thanks{The
    authors are with the Department of Mechanical and Aerospace
    Engineering, University of California, San Diego, USA, {\tt \small
      \{dmateosn,cortes\}@ucsd.edu}.}}

\maketitle

\begin{abstract}
  We present distributed subgradient methods for min-max problems with
  agreement constraints on a subset of the arguments of both the
  convex and concave parts.  Applications include
  %
  %
  constrained minimization problems where each constraint is a sum of
  convex functions in the local variables of the agents. In the latter
  case, the proposed algorithm reduces to primal-dual updates using
  local subgradients and Laplacian averaging on local copies of the
  multipliers associated to the global constraints.
  For the case of general convex-concave saddle-point problems, our analysis
  establishes the convergence of the running time-averages of the
  local estimates to a saddle point under periodic connectivity of the
  communication digraphs.  Specifically, choosing the gradient
  step-sizes in a suitable way, we show that the evaluation error is
  proportional to~$1/\sqrt{t}$, where~$t$ is the iteration step.
  {\color{black} We illustrate our results in simulation for an
    optimization scenario with nonlinear constraints coupling the
    decisions of agents that cannot communicate directly.  }
\end{abstract}

\section{Introduction}
%
%
{\color{black} Saddle-point problems arise in constrained optimization
  via the Lagrangian formulation and, more generally, are equivalent
  to variational inequality problems.
  These formulations find applications in cooperative control of
  multi-agent systems, in machine learning and game theory, and in
  equilibrium problems in networked systems,} motivating the study of
distributed strategies that are guaranteed to converge, scale well
with the number of agents, and are robust against a variety of
failures and uncertainties. Our objective in this paper is to design
and analyze distributed algorithms to solve general convex-concave
saddle-point problems.

\subsubsection*{Literature review}
This work builds on three related areas: iterative methods for
saddle-point problems~\cite{KA-LH-HU:58,AN-AO:09-jota}, dual
decompositions for constrained optimization~\cite[Ch. 5]{NP-SB:13},
\cite{SB-NP-EC-BP-JE:11}, and consensus-based distributed optimization
algorithms; see,
e.g.,~\cite{BJ-TK-MJ-KHJ:08,AN-AO:09,JW-NE:11,MZ-SM:12,RC-GN-LS-DV:15,BG-JC:14-tac}
and references therein.  Historically, these fields have been driven
by the need of solving constrained optimization problems and by an
effort of parallelizing the
computations~\cite{JNT:84,JNT-DPB-MA:86,DPB-JNT:97}, leading to
consensus approaches that allow different processors with local
memories to update the same components of a vector by averaging their
estimates.  Saddle-point or min-max problems arise in optimization
contexts such as worst-case design, exact penalty functions, duality
theory, and zero-sum games, see e.g.~\cite{DPB-AN-AEO:03},
{\color{black} and are equivalent to the variational inequality
  framework~\cite{GS-DPP-FF-JSP:10}, which includes as particular
  cases constrained optimization and many other equilibrium models
  relevant to networked systems, including traffic~\cite{SD:80} and
  supply chain~\cite{AN-MY-AHM-LSN:10}.} In a centralized scenario,
the work~\cite{KA-LH-HU:58} studies iterative subgradient methods to
find saddle points of a Lagrangian function and establishes
convergence to an arbitrarily small neighborhood depending on the
gradient stepsize. Along these lines,~\cite{AN-AO:09-jota} presents an
analysis for general convex-concave functions and studies the
evaluation error of the running time-averages, showing convergence to
an arbitrarily small neighborhood assuming boundedness of the
estimates. In~\cite{AN-AO:09-jota,AN-AO:10-siam}, the boundedness of
the estimates in the case of Lagrangians is achieved using a truncated
projection onto a closed set that preserves the optimal dual set,
which~\cite{JBHU-CL:93} shows to be bounded when the strong Slater
condition holds. This bound on the Lagrange multipliers depends on
global information and hence must be known beforehand.

Dual decomposition methods for constrained optimization are the
melting pot where saddle-point approaches come together with methods
for parallelizing computations, like the alternating direction method
of multipliers~\cite{SB-NP-EC-BP-JE:11}.  These methods rely on a
particular approach to split a sum of convex objectives by introducing
agreement constraints on copies of the primal variable, leading to distributed
strategies such as 
%
%
distributed primal-dual subgradient
methods~\cite{JW-NE:11,BG-JC:14-tac} {\color{black} where the vector of
  Lagrange multipliers associated with the Laplacian's nullspace is
  updated by the agents using local communication.}  Ultimately, these
methods allow to distribute global constraints that are sums of convex
functions via \emph{agreement on the
  multipliers}~\cite{MB-GN-FA:14,THC-AN-AS:14,AS-HJR:15}.
%
%
Regarding distributed constrained optimization, we highlight two
categories of constraints that determine the technical analysis and
the applications: 
%
%
%
the first type concerns a \emph{global} decision vector in which
agents need to agree, see,
e.g.,~\cite{DY-SX-HZ:11,MZ-SM:12,DY-DWCH-SX:15}, where all the agents
know the constraint, or see,
e.g.,~\cite{AN-AO-PAP:10,IN-ID-JAKS:10,MZ-SM:12}, where the constraint
is given by the intersection of abstract closed convex sets. The
second type \emph{couples} \emph{local} decision vectors across the
network, {\color{black} and is addressed by~\cite{DMA-TR-DS:10} with
  linear equality constraints, by~\cite{MB-GN-FA:14} with linear
  inequalities, by~\cite{THC-AN-AS:14} with inequalities given by the
  sum of convex functions on local decision vectors, where each one is
  only known to the corresponding agent, and by~\cite{AS-HJR:15} with
  semidefinite constraints.  }
The work~\cite{DMA-TR-DS:10} considers a distinction, that we also
adopt here, between \emph{constraint graph} {\color{black} (where edges
  arise from participation in a constraint)} and \emph{communication
  graph},
generalizing other paradigms where each agent needs to communicate
with all other agents involved in a particular
constraint~\cite{GN-FB:11,DR-JC:15-tac}.
%
{\color{black} When applied to distributed optimization, our work
  considers both kinds of constraints, and along
  with~\cite{DMA-TR-DS:10,MB-GN-FA:14,THC-AN-AS:14,AS-HJR:15}, has the
  crucial feature that agents participating in the same constraint are
  able to coordinate their decisions without direct
  communication. This approach has been successfully applied to control
  of camera networks~\cite{AAM-CD-AKRC-JAF:14} and decomposable
  semidefinite programs~\cite{AK-JL:15}.
}
%
%
This is possible using a strategy that allows an agreement condition
to play an independent role on a subset of \emph{both} primal and dual
variables.  {\color{black} Our novel contribution tackles} these
constraints from a more general perspective, namely, we provide a
multi-agent distributed approach for the \emph{general saddle-point
  problems} under an additional agreement condition on a subset of the
variables of both the convex and concave parts.  We do this by
combining the saddle-point subgradient methods
in~\cite[Sec. 3]{AN-AO:09-jota}
and the kind of linear proportional feedback on the disagreement
typical of consensus-based approaches, see
e.g.,~\cite{BJ-TK-MJ-KHJ:08,AN-AO:09,MZ-SM:12}, in distributed convex
optimization.  The resulting family of algorithms {\color{black} solve
  more general saddle-point problems than existing algorithms in the
  literature in a decentralized way, and also } particularize to a
novel class of primal-dual consensus-based subgradient methods when
\emph{the convex-concave} function is the \emph{Lagrangian} of the
minimization of a sum of convex functions under a constraint of the
same form.  In this particular case, the recent
work~\cite{THC-AN-AS:14}
uses primal-dual perturbed methods which enhance subgradient
algorithms by evaluating the latter at precomputed arguments called
\emph{perturbation points}.
These auxiliary computations require additional subgradient methods or proximal
methods that add to the computation and the communication complexity.
{\color{black} Similarly, the work~\cite{AS-HJR:15} considers
  primal-dual methods, where each agent performs a minimization of the
  local component of the Lagrangian with respect to its primal
  variable (instead of computing a subgradient step).
  Notably, this work makes explicit the treatment of semidefinite
  constraints.
  The work~\cite{MB-GN-FA:14} applies the \emph{Cutting-Plane
    Consensus} algorithm to the dual optimization problem under linear
  constraints. The decentralization feature is the same but the
  computational complexity of the local problems grows with the number
  of agents. }
%
{\color{black} The generality of our approach stems from the fact that}
our saddle-point strategy is applicable beyond the case of Lagrangians
in constrained optimization. In fact, we have recently considered
in~\cite{DMN-JC:15-necsys} distributed optimization problems with
nuclear norm regularization via a min-max formulation of the nuclear
norm where the convex-concave functions involved have, unlike
Lagrangians, a quadratic concave part.

\subsubsection*{Statement of contributions}

We consider general saddle-point problems with explicit agreement
constraints on a subset of the arguments of both the convex and
concave parts.  These problems appear in dual decompositions of
constrained optimization problems, 
and in other saddle-point problems where the convex-concave functions,
unlike Lagrangians, are not necessarily linear in the arguments of the
concave part.  {\color{black} This is a substantial improvement
  over prior work that only focuses on dual decompositions of
  constrained optimization. }
When considering constrained optimization problems, the agreement
constraints are introduced as an artifact to distribute both primal
and dual variables independently.  For instance, separable constraints
can be decomposed using agreement on dual variables, while a subset of
the primal variables can still be subject to agreement or eliminated
through Fenchel conjugation; local constraints can be handled through
projections; and part of the objective can be expressed as a
maximization problem in extra variables.  Driven by these important
classes of problems, {\color{black} our main contribution is } the
design and analysis of distributed coordination algorithms to solve
{\color{black} general convex-concave saddle-point problems } with agreement
constraints, {\color{black} and to do so with subgradient methods,
  which have less computationally complexity.} The coordination
algorithms that we study can be described as projected saddle-point
subgradient methods with Laplacian averaging, which naturally lend
themselves to distributed implementation.  For these algorithms we
characterize the asymptotic convergence properties in terms of the
network topology and the problem data, and provide the convergence
rate.  The technical analysis entails computing bounds on the
saddle-point evaluation error in terms of the disagreement, the size
of the subgradients, the size of the states of the dynamics, and the
subgradient stepsizes.  Finally, under assumptions on the boundedness
of the estimates and the subgradients, we further bound the cumulative
disagreement under joint connectivity of the communication graphs,
regardless of the interleaved projections, and make a choice of
decreasing stepsizes that guarantees convergence of the evaluation
error as~$1/\sqrt{t}$, where $t$ is the iteration step.  We
particularize our results to the case of distributed constrained
optimization with objectives and constraints that are a sum of convex
functions coupling local decision vectors across a network.  For this
class of problems, we also present a distributed strategy that lets
the agents compute a bound on the optimal dual set.  This bound
enables agents to project the estimates of the multipliers onto a
compact set (thus guaranteeing the boundedness of the states and
subgradients of the resulting primal-dual projected subgradient
dynamics) in a way that preserves the optimal dual set.  Various
simulations illustrate our results.
%

%
%

\section{Preliminaries}\label{sec:preliminaries-multitask}

Here we introduce basic notation and notions from graph theory and
optimization used throughout the paper.

\subsection{Notational conventions}

We denote by~$\real^n$ the $n$-dimensional Euclidean space, by
$\identity_n\in\realmatricesarg{n}$ the identity matrix in~$\real^n$,
and by~$\ones_n \in \real^n$ the vector of all ones.  Given two
vectors, $u$, $v\in\real^n$, we denote by $u\ge v$ the entry-wise set
of inequalities $u_i\ge v_i$, for each $i=1,\dots, n$.  Given a vector
$v\in\real^n$, we denote its Euclidean norm, or two-norm, by
$\norm{v}=\sqrt{\sum_{i=1}^n v_i^2}$ and the one-norm by
$\normone{v}=\sum_{i=1}^n \absolute{v_i}$.  Given a convex set
$\mathcal{S}\subseteq\real^n$, a function
$\map{f}{\mathcal{S}}{\real}$ is convex if $f(\alpha x + (1-\alpha) y)
\le \alpha f(x) + (1-\alpha) f(y)$ for all $\alpha \in [0,1]$ and $x,y
\in \mathcal{S}$.  A vector $\xi_{x}\in\real^n$ is a subgradient of
$f$ at $x\in\mathcal{S}$ if $ f(y)-f(x)\ge\xi_{x}\tp(y-x)$, for all
$y\in\mathcal{S}$.  We denote by $\partial f(x)$ the set of all such
subgradients. The function $f$ is concave if $-f$ is convex.  A vector
$\xi_{x}\in\real^n$ is a subgradient of a concave function $f$ at
$x\in\mathcal{S}$ if $-\xi_{x} \in \partial (-f)(x)$.  Given a \emph{closed}
convex set~$\mathcal{S}\subseteq\real^n$, the orthogonal projection
$\mathcal{P}_{\mathcal{S}}$ onto~$\mathcal{S}$ is
\begin{align}\label{eq:def-projection}
  \projec{\mathcal{S}}{x}\in\arg\min_{x'\in \mathcal{S}} \norm{x-x'} .
\end{align} {\color{black} This value exists and is unique. (Note that
  \emph{compactness} could be assumed without loss of generality
  taking the intersection of $\mathcal{S}$ with balls centered at
  $x$.)  We use the following basic property of the orthogonal
  projection: for every $x\in\mathcal{S}$ and $x'\in\real^n$,
  \begin{align}\label{eq:prelims-projection-property}
    \big(\projec{\mathcal{S}}{x'}-x'\big)(x'-x)\le&\, 0 .
  \end{align}
}
For a symmetric matrix $A\in\real^{n\times n}$, we denote by
$\lambdamin(A)$ and $\lambdamax(A)$ its minimum and maximum
eigenvalues, {\color{black} and for any matrix $A$, we denote by
  $\sigmamax(A)$ its maximum singular value.  } We use $\otimes$ to
denote the Kronecker product of matrices.
%
%

\subsection{Graph theory}

We review basic notions from graph theory
following~\cite{FB-JC-SM:08cor}.  A (weighted) digraph $\graph \defin
(\vertexset,\edgeset,\Adj)$ is a triplet where $\vertexset \defin
\until{N}$ is the vertex set, $\edgeset \subseteq \vertexset \times
\vertexset$ is the edge set, and $\Adj \in \realnonnegmatricesarg{N}$
is the weighted adjacency matrix with the property that
$\adj_{ij}\defin A_{ij}>0$ if and only if $(i,j) \in \edgeset$.  The
complete graph is the digraph with edge set
$\vertexset \times \vertexset$.
Given $\graph_1 = (\vertexset,\edgeset_1,\Adj_1)$ and $\graph_2 =
(\vertexset,\edgeset_2,\Adj_2)$, their union is the digraph $\graph_1
\cup \graph_2 = (\vertexset,\edgeset_1 \cup \edgeset_2 ,\Adj_1+
\Adj_2) $.  A path is an ordered sequence of vertices such that any
pair of vertices appearing consecutively is an edge.  A digraph is
strongly connected if there is a path between any pair of distinct
vertices.  A sequence of digraphs $\big\{\graph_t\defin
(\vertexset,\edgeset_t, \Adj_t)\big\}_{t\ge 1}$ is
$\degn$-nondegenerate, for $\degn\in\realpositive$, if the weights are
uniformly bounded away from zero by $\degn$ whenever positive, i.e.,
for each $t\in\integerspositive$, $ \adj_{ij,t}\defin(\Adj_t)_{ij}
>\degn$ whenever ${\adj_{ij,t}}>0$.  A sequence $\{\graph_t\}_{t\ge1}$
is $B$-jointly connected, for $B\in\integerspositive$, if for each
$k\in\integerspositive$, the digraph $\graph_{kB} \cup \dots \cup
\graph_{(k+1)B-1}$ is strongly connected.  The Laplacian matrix
$\lap\in\realmatricesarg{N}$ of a digraph $\graph$ is $\lap \defin
\diag(\Adj\ones_N)-\Adj$. Note that $\lap \ones_N = \zeros{N}$.  The
weighted out-degree and in-degree of $i \in \vertexset$ are,
respectively, $ \dout(i) \defin\sum_{j=1}^{N}\adj_{ij} $ and
$\din(i)\defin\sum_{j=1}^N \adj_{ji} $.  A digraph is weight-balanced
if $\dout(i) =\din(i) $ for all $ i \in \vertexset$, that is,
$\ones_N\tp \lap = \zeros{N}$.
For convenience, we let $\lapk \defin \identity_N -\tfrac{1}{N}
\ones_N \ones_N\tp$ denote the Laplacian of the complete graph with
edge weights~$1/N$. Note that $\lapk$ is idempotent, i.e.,
$\lapk^2=\lapk$.
For the sake of the reader, Table~\ref{tab:matrices-definitions-saddle-methods} 
collects some shorthand notation.
\begin{table}[h]
  \centering
  \begin{tabular} {|c|c|c|}
    \hline
    $\matrixM=\tfrac{1}{N} \ones_N\ones_N\tp$
    & $\lapk=\identity_N -\matrixM$ 
    & $\tlap= \diag(\Adj_t\ones_N)-\Adj_t$ 
    \\
    \hline
    $\mb=\matrixM\otimes\identity_d$ 
    & $\lapkb=\lapk\otimes\identity_d$ 
    & $\tlapb= \tlap\otimes\identity_d$
    \\
    \hline
  \end{tabular}
  \caption{Notation for graph matrices
    employed along the paper, where the dimension $d$ depends on the context.
  }\label{tab:matrices-definitions-saddle-methods} 
\end{table}

\subsection{Optimization and saddle
  points}\label{sec:prelims-optimization}

For any function~$\map{\lag}{\wset\times\mathcal{M}}\real$, the
\emph{max-min inequality}~\cite[Sec 5.4.1]{SB-LV:09} states that
\begin{align}\label{eq:min-max-inequality}
  \inf_{w\in\wset}\sup_{\mu\in\mathcal{M}} \lag(w,\mu) \ge
  \sup_{\mu\in\mathcal{M}}\inf_{w\in\wset} \lag(w,\mu) .
\end{align}
When equality holds, we say that~$\lag$ satisfies the \emph{strong
  max-min property} (also called the \emph{saddle-point} property).  A point
$(w^*,\mu^*)\in\wset\times\mathcal{M}$ is called a \emph{saddle point}
if
\begin{align*} {w}^*=\inf_{w\in\wset}\lag(w,{\mu}^*) \; \text{and} \;
  {\mu}^*=\sup_{\mu\in\mathcal{M}}\lag({w}^*,\mu) .
\end{align*}
\cite[Sec. 2.6]{DPB-AN-AEO:03} discusses sufficient conditions to
guarantee the existence of saddle points. Note that the existence of
saddle points implies the strong max-min property.
  Given functions
$\map{f}{\real^n}{\real}$, $\map{g}{\real^{m}}{\real}$ and
$\map{h}{\real^{p}}{\real}$, the \emph{Lagrangian} for the problem
\begin{align}\label{eq:prelimns-def-optimization-problem}
  \min_{w\in\real^n} f(w) \quad \text{s.t.}\quad g(w)\le 0,\; h(w)= 0,
\end{align}
is defined as 
\begin{align}\label{eq:prelimns-Lagrangian-def}
\lag(w,\mu,\lambda)=f(w)+{\mu}\tp g(w)+{\lambda}\tp h(w)
\end{align}
for $(\mu,\lambda)\in\realnonnegative^{m}\times\real^{p}$.
In this case, inequality~\eqref{eq:min-max-inequality} is called
\emph{weak-duality}, and if equality holds, then we say
that~\emph{strong-duality} (or Lagrangian duality) holds.
%
%
If a point $(w^*,\mu^*,\lambda^*)$ is a saddle point for the
Lagrangian, then $w^*$ solves the constrained minimization
problem~\eqref{eq:prelimns-def-optimization-problem} and
$(\mu^*,\lambda^*)$
solves the \emph{dual problem}, which is maximizing the \emph{dual
  function} $q(\mu,\lambda)\defin\inf_{w\in\real^n}
\lag(w,\mu,\lambda)$ over $\realnonnegative^{m}\times\real^{p}$.
{\color{black}
  This implication is part of the \textit{Saddle Point Theorem}. (The
  reverse implication establishes the existence of a saddle-point
  --and thus strong duality-- adding a \textit{constraint
    qualification} condition.)  Under the saddle-point condition, the
  \emph{optimal dual vectors} $(\mu^*,\lambda^*)$ coincide with the
  \emph{Lagrange multipliers}~\cite[Prop. 5.1.4]{DPB:99}.  }
In the case of affine linear constraints, the dual function can be
written using the \emph{Fenchel conjugate} of~$f$,
defined in $\real^n$ as
\begin{align}\label{eq:prelims-Fenchel-def}
f\fen(x)\defin\sup_{w\in\real^n} \{ x\tp w-f(w)\} .
\end{align}

\section{Distributed algorithms for saddle-point problems under
  agreement constraints}

This section describes the problem of interest.  Consider closed
convex sets $\wnsetan\subseteq\realdimw$, $\dmset\subseteq\realdimd$,
$\munsetan\subseteq\realdimmu$, $\zset\subseteq\realdimz$ and a
function $\map{\phicc}{\wnsetan\times \dnmset\times\munsetan
  \times\znset}{\real}$ which is jointly convex on the first two
arguments and jointly concave on the last two arguments. We seek to
solve the constrained saddle-point problem:
\begin{align}\label{eq:min-max-convex-concave-function}
  \min_{\substack{\wn\in\wnsetan,\, \dnm\in\dnmset
      \\
      \dm^i=\dm^j,\,\forall i,j} }
  \max_{\substack{\mun\in\munsetan,\,\zn\in\znset
      \\
      z^i=z^j,\,\forall i,j} } \;\phicc(\wn,\dnm,\mun,\zn) ,
\end{align}
where $\dnm \defin (\dm^1,\dots, \dm^N)$ and $\zn \defin (z^1,\dots,
z^N)$.  The motivation for distributed algorithms and the
consideration of explicit agreement constraints
in~\eqref{eq:min-max-convex-concave-function} comes from {\color{black}
  decentralized or parallel computation approaches} in network
optimization and machine learning.  In such scenarios, global decision
variables, which {\color{black} need to be determined from the
  aggregation of local data}, can be duplicated into distinct ones so
that each agent has its own local version to operate with. Agreement
constraints are then imposed across the network to ensure the
equivalence to the original optimization problem.  We explain this
procedure next, specifically through the dual decomposition of
optimization problems where objectives and constraints are a sum of
convex functions.




\subsection{Optimization problems with separable
  constraints}\label{sec-linear-semidefinite-constraints-intro}

We illustrate here how optimization problems with constraints given by
a sum of convex functions can be reformulated in the
form~\eqref{eq:min-max-convex-concave-function} to make them amenable
to distributed algorithmic solutions.  Our focus are constraints
coupling the local decision vectors of agents that cannot communicate
directly.

Consider a group of agents $\until{N}$, and let
$\map{f^i}{\real^{n_i}\times\realdimd}{\real}$ and the components of
$\map{g^i}{\real^{n_i}\times\realdimd}{\real^m}$ be convex functions
associated to agent $i\in\until{N}$. These functions depend on both a
local decision vector $w^i\in\wseti$, with
$\wseti\subseteq\real^{n_i}$ convex, and on a global decision
vector~$\dm\in\dmset$, with~$\dmset\subseteq\realdimd$ convex.  The
optimization problem reads as
\begin{align}\label{eq:problem-linear-inequality-constraints}
  \min_{\substack{w^i\in\wseti,\,\forall i
      \\
      \dm\in\dmset } }&\, \sum_{i=1}^N f^i(w^i, \dm) \notag
  \\
  \rm{s.t.}&\, g^1(w^1,\dm)+\dots+g^N(w^N,\dm)\le 0 .
\end{align}
This problem can be reformulated as a constrained saddle-point problem
as follows. We first construct the corresponding Lagrangian
function~\eqref{eq:prelimns-Lagrangian-def} and introduce copies
$\{z^i\}_{i=1}^N$ of the Lagrange multiplier $z$ associated to the
global constraint in~\eqref{eq:problem-linear-inequality-constraints},
then associate each $z^i$ to $g^i$, and impose the agreement
constraint $z^i=z^j$ for all $i$, $j$.  Similarly, we also introduce
copies $\{\dm^i\}_{i=1}^N$ of the global decision vector $\dm$ subject
to agreement, $\dm^i=\dm^j$ for all $i,j$.  The existence of a saddle
point implies that strong duality is attained and there exists a
solution of the
optimization~\eqref{eq:problem-linear-inequality-constraints}.
%
Formally,
\begin{subequations}\label{eq:separable-constraints-all}
  \begin{align}
    &\,\min_{\substack{
        w^i\in\wseti
        \\
        \dm\in\dmset}}\;\;\max_{z\in \realnonnegative^{m}} 
    \sum_{i=1}^N f^i(w^i,\dm)
    +z\tp \sum_{i=1}^Ng^{i}(w^i,\dm)
    \\
    =
    &\, \min_{
    \substack{
      w^i\in\wseti
      \\
      \dm\in\dmset}}\;\;
  \max_{\substack{z^i\in \realnonnegative^{m}\;\;
      \\
      z^i=z^j,\,\forall i,j
    }}\;\;
  \sum_{i=1}^N \big( f^i(w^i,\dm)
  + {z^i}\tp g^{i}(w^i,\dm) \big)
  \label{eq:separable-constraints-extra-dm}
  \\
  =
  &\, \min_{
    \substack{w^i\in\wseti
      \\
      \dm^i\in\dmset
      \\
      \dm^i=\dm^j,\,\forall i,j
    }}\;\;
  \max_{\substack{z^i\in \realnonnegative^{m}\;\;
      \\
      z^i=z^j,\,\forall i,j
    }}\;\;
  \sum_{i=1}^N \big( f^i(w^i,\dm^i)
  + {z^i}\tp g^{i}(w^i,\dm^i) \big) .
  \label{eq:separable-constraints-c}
\end{align}
\end{subequations}
This formulation has its roots in the classical dual decompositions
surveyed in~\cite[Ch. 2]{SB-NP-EC-BP-JE:11}, see
also~\cite[Sec. 1.2.3]{AN-AO:10} and~\cite[Sec. 5.4]{NP-SB:13} for the
particular case of resource allocation.
While~\cite{SB-NP-EC-BP-JE:11,AN-AO:10} suggest to broadcast a
centralized update of the multiplier, and the method
in~\cite{NP-SB:13} has an implicit projection onto the probability
simplex, the formulation~\eqref{eq:separable-constraints-all} has the
multiplier associated to the global constraint estimated in a
decentralized way.  The recent
works~\cite{MB-GN-FA:14,THC-AN-AS:14,AS-HJR:15} implicitly rest on the
above formulation of \emph{agreement on the multipliers} {\color{black}
  Section~\ref{sec:application-constrained-optimization}
  particularizes our general saddle-point strategy to these
  distributed scenarios.  }


\begin{remark}\longthmtitle{Distributed formulations via Fenchel
    conjugates}\label{re:Fenchel}
  {\rm To illustrate the generality of the min-max
    problem~\eqref{eq:separable-constraints-c}, we show here how
    \emph{only} the particular case of \emph{linear} constraints can
    be reduced to a maximization problem under
    agreement. Consider the particular case of $\min_{w^i\in\real^{n_i}}
    \sum_{i=1}^N f^i(w^i)$, subject to a linear constraint
    \begin{align*}
      \sum_{i=1}^N A^i w^i-b\le 0 ,
    \end{align*}
    with $A^i\in\real^{m\times n_i}$ and $b\in\real^m$. The above
    formulation suggests a distributed strategy that \emph{eliminates}
    the primal variables using Fenchel conjugates~\eqref{eq:prelims-Fenchel-def}.
  %
  %
  Taking~$\{b^{i}\}_{i=1}^N$ such that $\sum_{i=1}^N b^{i}=b$,
  %
  this problem can be transformed, if a saddle-point exists (so that
  strong duality is attained), into
  %
  %
  \begin{subequations}\label{eq:Fenchel-formulation-distributed}
    \begin{align}
      &\,\max_{z\in \zset} \min_{w^i\in\real^{n_i},\,\forall i}
      \sum_{i=1}^N f^i(w^i)+\sum_{i=1}^N (z\tp A^{i}w^i -z\tp b^i)
      \\
      =&\,\max_{z\in \zset}\sum_{i=1}^N \big( -{f^i}\fen( -{A^{i}}\tp
      z)-z\tp b^i\big)
      \label{eq:Fenchel-formulation-distributed-b} 
      \\
      =&\,\max_{\substack{z^i\in \zset,\,\forall i
          \\
          z^i=z^j,\,\forall i,j
        }}
      \sum_{i=1}^N \big(-{f^i}\fen( -{A^{i}}\tp z^i)-{z^i}\tp b^i\big) ,
      \label{eq:Fenchel-formulation-distributed-c} 
    \end{align}
  \end{subequations}
  where $\zset$ is either $\real^m$ or $\realnonnegative^m$ depending
  on whether we have equality or inequality ($\le$) constraints
  in~\eqref{eq:problem-linear-inequality-constraints}.  By~\cite[Prop.
  11.3]{RTR-RJBW:98}, the optimal primal values can be recovered
  locally as
  \begin{align}\label{eq:def-primal-recovery}
    {w^i}^*\defin\partial {f^i}\fen(-{A^i}\tp {z^i}^*) ,\qquad
    i\in\until{N}
  \end{align}
  without extra communication. {\color{black} Thus, our strategy
    generalizes the class of convex optimization problems with linear
    constraints studied in~\cite{DMA-TR-DS:10}, which distinguishes
    between the \emph{constraint graph} (where edges arise from
    participation in a constraint) and the \emph{network graph}, and
    defines \emph{distributed} with respect to the latter.  }}
\oprocend
\end{remark}

\subsection{Saddle-point dynamics with Laplacian averaging}

We propose a projected subgradient method to solve constrained
saddle-point problems of the
form~\eqref{eq:min-max-convex-concave-function}. The agreement
constraints are addressed via Laplacian averaging, allowing the design
of distributed algorithms \emph{when} the convex-concave functions are
separable as in
Sections~\ref{sec-linear-semidefinite-constraints-intro}.
{\color{black} The generality of this dynamics is inherited by the
  general structure of the convex-concave min-max
  problem~\eqref{eq:min-max-convex-concave-function}. We have chosen
  this structure both for convenience of analysis, from the
  perspective of the saddle-point evaluation error, and, more
  importantly, because it allows to model problems beyond constrained
  optimization; see, e.g.,~\cite{GS-DPP-FF-JSP:10} regarding the
  variational inequality framework, which is equivalent to the
  saddle-point framework.
  Formally, the dynamics is}
\begin{subequations}\label{eq:minmax-dyn-wnxnzn}
  \begin{align}
    \wntphat &=\wnt-\gradstept \subgwnt\label{eq:minmax-dyn-wn}
    \\
    \dnmtphat& =\dnmt-\cstep\tlapbk\dnmt-\gradstept
    \subgdnmt\label{eq:minmax-dyn-xn}
    \\
    \muntphat&= \munt+\gradstept \subgmunt\label{eq:minmax-dyn-mun}
    \\
    \zntphat &= \znt-\cstep\tlapb\znt+\gradstept
    \subgznt\label{eq:minmax-dyn-zn}
    \\
    (\wntp,\dnmtp,&\,\muntp,\zntp)=\projec{\cartnset}
    {\wntphat,\dnmtphat,\muntphat,\zntphat} ,\notag
  \end{align}
\end{subequations}
where $\tlapb=\tlap\otimes\identity_{\dimdn}$ or
$\tlapb=\tlap\otimes\identity_{\dimzn}$, depending on the context,
with $\tlap$ the Laplacian matrix of $\graph_t$;
$\cstep\in\realpositive$ is the consensus stepsize,
$\{\gradstept\}_{t\ge 1}\subset\realpositive$ are the learning rates;
\begin{align*}
  \subgwnt\in&\,\partial_{\wn}\phicc(\wnt,\dnmt,\munt,\znt) ,
  \\
  \subgdnmt\in&\,\partial_{\dnm}\phicc(\wnt,\dnmt,\munt,\znt) ,
  \\
  \subgmunt\in&\,\partial_{\mun}\phicc(\wnt,\dnmt,\munt,\znt) ,
  \\
  \subgznt\in&\,\partial_{\zn}\phicc(\wnt,\dnmt,\munt,\znt) ,
\end{align*}
and 
$\mathcal{P}_{\cartnset}$ represents the orthogonal projection onto
the closed convex set~$\cartnset\defin \wnsetan
\times\dnmset\times\munsetan\times\znset$ as defined
in~\eqref{eq:def-projection}.  This family of algorithms particularize
to a novel class of primal-dual consensus-based subgradient methods
\emph{when} the convex-concave function takes the Lagrangian form
discussed in Section~\ref{sec-linear-semidefinite-constraints-intro}.
{\color{black} 
%
  In general, the dynamics~\eqref{eq:minmax-dyn-wnxnzn} goes beyond
  any specific multi-agent model.  However, when interpreted in this
  context, the Laplacian component corresponds to the model for the
  interaction among the agents.}
In the upcoming analysis, we make network considerations that affect
the evolution of~$\tlapbk\dnmt$ and~$\tlapb\znt$, {\color{black} which
  measure the disagreement among the corresponding components of
  $\dnmt$ and $\znt$ via the Laplacian of the time-dependent adjacency
  matrices.} 
These quantities are amenable for distributed computation, i.e., the
computation of the $i$th block requires {\color{black} the blocks
  $\dm_t^j$ and $z_t^j$ of the } network variables corresponding to
indexes $j$ with $\adj_{ij,t}\defin(\Adj_t)_{ij} >0$.
On the other hand, whether the subgradients
in~\eqref{eq:minmax-dyn-wnxnzn} can be computed with \emph{local}
information \emph{depends} on the structure of the function $\phicc$
in~\eqref{eq:min-max-convex-concave-function} in the context of a given
networked problem.
Since this issue is anecdotal for our analysis, for the sake of
generality we consider a general convex-concave function~$\phicc$.




\section{Convergence analysis}\label{sec:convergence-analysis}

Here we present our technical analysis on the convergence properties
of the dynamics~\eqref{eq:minmax-dyn-wnxnzn}.  Our starting point is
the assumption that a solution
to~\eqref{eq:min-max-convex-concave-function} exists, namely, a saddle
point~$(\wnstarwo, \dnmstar, \munstar,\znstarwo)$ of~$\phicc$ on
$\cartnset \defin \wnsetan\times\dnmset\times\munsetan\times\znset$
under the agreement condition on $\dnmset$ and $\znset$. That is, with
$\dnmstar = \dm^*\otimes\ones_N$ and $\znstarwo = z^*\otimes\ones_N$
for some $(\dm^*,z^*) \in \dmset\times\zset$.  {\color{black} (We
  cannot actually conclude the feasibility property of the original
  problem \emph{from} the evolution of the estimates.)}
%
%
We then study the evolution of the \emph{running time-averages}
{\color{black} (sometimes called \emph{ergodic sums}; see,
  e.g.,~\cite{AS-HJR:15}) }
\begin{alignat*}{2}
  \wntpav & =\frac{1}{t}\sum_{s=1}^{t} \wns , & \quad \dnmtpav&
  =\frac{1}{t}\sum_{s=1}^{t} \dnms ,
  \\
  \muntpav & =\frac{1}{t}\sum_{s=1}^{t} \muns , & \quad \zntpav &
  =\frac{1}{t}\sum_{s=1}^{t} \zns .
\end{alignat*}

We summarize next our overall strategy to provide the reader with
a \emph{roadmap} of the forthcoming analysis.  In
Section~\ref{sec:saddle-point evaluation error}, we bound the
saddle-point evaluation error
\begin{align}\label{eq:saddle-point-error-sketch}
  t\phicc(\wntpav,\dnmtpav,\muntpav,\zntpav)
  -t\phicc(\wnstarwo,\dnmstar,\munstar,\znstarwo) .
\end{align}
in terms of the following quantities: the initial conditions, the size
of the states of the dynamics, the size of the subgradients, and the
cumulative disagreement of the running time-averages.  Then, in
Section~\ref{sec:cum-disagreement} we bound the cumulative
disagreement in terms of the size of the subgradients and the learning
rates.  Finally, in
Section~\ref{sec:proof-main-result-saddle-partial-agreement} we
establish the saddle-point evaluation convergence result using the
assumption that the estimates generated by the
dynamics~\eqref{eq:minmax-dyn-wnxnzn}, as well as the subgradient
sets, are uniformly bounded.  (This assumption can be met in
applications by designing projections that preserve the saddle points,
particularly in the case of distributed constrained optimization that
we discuss later.)  In our analysis, we conveniently choose the
learning rates~$\{\gradstept\}_{t\ge 1}$ using the Doubling Trick
scheme~\cite[Sec. 2.3.1]{SSS:12} to find lower and upper bounds
on~\eqref{eq:saddle-point-error-sketch} proportional to~$\sqrt{t}$.
Dividing by~$t$ finally allows us to conclude that the saddle-point
evaluation error of the running time-averages is bounded
by~${1}/{\sqrt{t}}$.

\subsection{Saddle-point evaluation error in terms of the
  disagreement}\label{sec:saddle-point evaluation error}

Here, we establish the saddle-point evaluation error of the running
time-averages in terms of the disagreement.  Our first result, whose
proof is presented in the Appendix, establishes a pair of inequalities
regarding the evaluation error of the states of the dynamics with
respect to a generic point in the variables of the convex and concave
parts, respectively.



\begin{lemma}\longthmtitle{Evaluation error of the states in terms of
    the
    disagreement}\label{le:basic-bound-convex-concave-function-differences}
  Let the sequence $\{(\wnt,\dnmt,\munt,\znt)\}_{t\ge 1}$ be generated
  by the coordination algorithm~\eqref{eq:minmax-dyn-wnxnzn} over a
  sequence of arbitrary weight-balanced
  digraphs~$\{\graph_t\}_{t\ge1}$ such that {\color{black} $\sup_{t\ge
      1} \sigmamax(\tlap)\le\lambdaup$,} and with
  \begin{align}\label{eq:upper-cons-cstep}
    \cstep\le\big( \max\setdefbig{\doutt(k)}{k\in\vertexset,\:
      t\in\integerspositive} \big)^{-1} .
  \end{align}
  Then, for any sequence of learning rates
  $\{\gradstept\}_{t\ge1}\subset\realpositive$ and any
  $(\wnp,\dnmp)\in\wnsetan\times\dnmset$,
  the following holds:
  \begin{align}\label{eq:prop-difference-dnmt-xn}
    &\, 2(\phicc(\wnt,\dnmt,\munt,\znt)-\phicc(\wnp,\dnmp,\munt,\znt)) 
    \\
    \le&\, \tfrac{1}{\gradstept}
    \big(\norm{\wnt-\wnp}^2 -\norm{\wntp-\wnp}^2\big)
    \notag
    \\
    &\,+ \tfrac{1}{\gradstept}
    \big(\norm{\mbk\dnmt-\dnmp}^2 -\norm{\mbk\dnmtp-\dnmp}^2\big)
    \notag
    \\
    &\,+6 \gradstept \norm{\subgwnt}^2+6\gradstept \norm{\subgdnmt}^2
    \notag
    \\
    &\, 
    +2\norm{\subgdnmt}(2+\cstep\lambdaup)\norm{\lapkbk\dnmt}
    +2\norm{\subgdnmt}\norm{\lapkbk\dnmp} .
    \notag         
  \end{align}
  Also, for any $(\munp, \znp)\in\munsetan\times\znset$, the analogous holds,
  \begin{align}\label{eq:prop-difference-znt-zn}
    &\, 2(\phicc(\wnt,\dnmt,\munt,\znt)-\phicc(\wnt,\dnmt,\munp,\znp))
    \\
    \ge &\, -\tfrac{1}{\gradstept}
    \big(\norm{\munt-\munp}^2 -\norm{\muntp-\munp}^2\big)
    \notag
    \\
    &\,-\tfrac{1}{\gradstept}
    \big(\norm{\mb\znt-\znp}^2 -\norm{\mb\zntp-\znp}^2\big)
    \notag
    \\
    &\,-6 \gradstept \norm{\subgmunt}^2-   6\gradstept \norm{\subgznt}^2
    \notag
    \\
    &\, -2\norm{\subgznt}(2+\cstep\lambdaup)\norm{\lapkb\znt}
    -2\norm{\subgznt}\norm{\lapkb\znp} .
    \notag              
  \end{align}
\end{lemma}

Building on
Lemma~\ref{le:basic-bound-convex-concave-function-differences}, we
next obtain bounds for the sum over time of the evaluation errors with
respect to a generic point and the running time-averages.

\begin{lemma}\longthmtitle{Cumulative evaluation error of the states
    with respect to running time-averages in terms of
    disagreement}\label{le:cummulative-concave-function-differences}
  Under the same assumptions of
  Lemma~\ref{le:basic-bound-convex-concave-function-differences},
  for any $(\wnp,\dnmp,\munp,\znp)\in
  \wnsetan\times\dnmset\times\munsetan\times\znset$,
  the difference
  \begin{align*} 
    \sum_{s=1}^{t}\phicc(\wns,\dnms,\muns,\zns) -
    t\phicc(\wnp,\dnmp,\muntpav,\zntpav)
  \end{align*}
  is upper-bounded by $\frac{\upperb(t,\wnp,\dnmp)}{2}$, while the
  difference
  \begin{align*}
    \sum_{s=1}^{t}\phicc(\wns,\dnms,\muns,\zns) -
    t\phicc(\wntpav,\dnmtpav,\munp,\znp) \notag
  \end{align*}
  is lower-bounded by $-\frac{\upperb(t,\munp,\znp)}{2}$,
  where
  \begin{align}\label{eq:cummulative-difference-dnmt-xn}
    &\,\upperb(t,\wnp,\dnmp) \equiv
    \upperb\big(t,\wnp,\dnmp,\{\wns\}_{s=1}^t,\{\dnms\}_{s=1}^t\big)
    \\
    =&\,\sum_{s=2}^{t} \big(\norm{\wns -\wnp}^2+\norm{\mbk\dnms
      -\dnmp}^2\big)
    \big(\tfrac{1}{\gradsteps}-\tfrac{1}{\gradstepsm}\big)
    \notag
    \\
    &\,+\tfrac{2}{\gradstep_{1}} \big(\norm{\wn_{1}}^2+\norm{\wnp}^2
    +\norm{\dnm_{1}}^2 +\norm{\dnmp}^2\big) \nonumber
    \\
    &\,+6\sum_{s=1}^{t} \gradsteps (\norm{\subgwns}^2+
    \norm{\subgdnms}^2) \nonumber
    \\
    +&\,2(2+\cstep\lambdaup) \sum_{s=1}^{t}
    \norm{\subgdnms}\norm{\lapkbk\dnms}
    +2\norm{\lapkbk\dnmp}\sum_{s=1}^{t} \norm{\subgdnms} ,
  \end{align}
  and $\upperb(t,\munp,\znp)\equiv
  \upperb\big(t,\munp,\znp,\{\muns\}_{s=1}^t,\{\zns\}_{s=1}^t\big)$.
  %
\end{lemma}
\begin{proof}
  By adding~\eqref{eq:prop-difference-dnmt-xn} over $s=1,\dots,t$,
  we obtain
  \begin{align*}
    &\,
    2\sum_{s=1}^{t}(\phicc(\wns,\dnms,\muns,\zns)-\phicc(\wnp,\dnmp,\muns,\zns))
    \\
    \le &\,\sum_{s=2}^{t} \big(\norm{\wns -\wnp}^2+\norm{\mbk\dnms
      -\dnmp}^2\big)
    \big(\tfrac{1}{\gradsteps}-\tfrac{1}{\gradstepsm}\big)
    \nonumber
    \\
    &\,+\tfrac{1}{\gradstep_{1}} 
    \big(\norm{\wn_{1}-\wnp}^2+\norm{\mbk\dnm_{1}-\dnmp}^2\big)  
    \nonumber   
    \\
    &\,+6\sum_{s=1}^{t} \gradsteps (\norm{\subgwns}^2+ \norm{\subgdnms}^2) 
    \nonumber 
    \\
    +&\,2(2+\cstep\lambdaup) \sum_{s=1}^{t} \norm{\subgdnms}\norm{\lapkbk\dnms}
    +2\norm{\lapkbk\dnmp}\sum_{s=1}^{t} \norm{\subgdnms}.
    \nonumber
  \end{align*}
  This is bounded from above by~$\upperb(t,\wnp,\dnmp)$ because
  $\norm{\mbk\dnm_{1}-\dnmp}^2\le 2\norm{\dnm_{1}}^2
  +2\norm{\dnmp}^2$, which follows from the triangular inequality,
  Young's inequality, the sub-multiplicativity of the norm, and the
  identity $\norm{\mbk}= 1$. Finally, by the concavity of~$\phicc$ in
  the last two arguments,
  \begin{align*}
    \phicc(\wnp,\dnmp,\muntpav,\zntpav) \ge \frac{1}{t}
    \sum_{s=1}^{t}\phicc(\wnp,\dnmp, \muns,\zns),
  \end{align*}
  so the upper bound in the statement follows. Similarly, we obtain
  the lower bound by adding~\eqref{eq:prop-difference-znt-zn} over
  $s=1,\dots,t$
  and using that $\phicc$ is jointly convex in the first two
  arguments,
  \begin{align*}
    \phicc(\wntpav,\dnmtpav,\muns,\zns)
    \le\frac{1}{t}\sum_{s=1}^{t}\phicc(\wns,\dnms,\muns,\zns) ,
  \end{align*}
  which  completes the proof.
\end{proof}

The combination of the pair of inequalities in
Lemma~\ref{le:cummulative-concave-function-differences} allows us to
derive the saddle-point evaluation error of the running time-averages
in the next result.

\begin{proposition}\longthmtitle{Saddle-point evaluation error of
    running time-averages}\label{prop:approx-saddle-points}
  Under the same hypotheses of
  Lemma~\ref{le:basic-bound-convex-concave-function-differences}, for
  any saddle point~$(\wnstarwo, \dnmstar, \munstar,\znstarwo)$
  of~$\phicc$ on $\wnsetan\times\dnmset\times\munsetan\times\znset$
  with $\dnmstar=\dm^*\otimes\ones_N$ and
  $\znstarwo=z^*\otimes\ones_N$ for some
  $(\dm^*,z^*)\in\dmset\times\zset$,
  the following holds:
  \begin{align}\label{eq:general-bound-lower-upper}
    &\,-\upperb(t,\munstar,\znstarwo)-\upperb(t,\wntpav,\dnmtpav)
    \notag
    \\
    \le&\,2t\phicc(\wntpav,\dnmtpav,\muntpav,\zntpav)
    -2t\phicc(\wnstarwo,\dnmstar,\munstar,\znstarwo) \notag
    \\
    \le&\,
    \upperb(t,\wnstarwo,\dnmstar)+\upperb(t,\muntpav,\zntpav)\,.
  \end{align}
\end{proposition}
\begin{proof}
  We show the result in two steps, by evaluating the bounds from
  Lemma~\ref{le:cummulative-concave-function-differences} in two sets
  of points and combining them.  First, choosing
  $(\wnp,\dnmp,\munp,\znp)=(\wnstarwo,\dnmstar,\munstar,\znstarwo)$ in
  the bounds of
  Lemma~\ref{le:cummulative-concave-function-differences}; invoking
  the saddle-point relations
  \begin{align*}
    \phicc(\wnstarwo,\dnmstar,\muntpav,\zntpav)&\,
    \le\phicc(\wnstarwo,\dnmstar,\munstar,\znstarwo)
    \\
    &\,\le\phicc(\wntpav,\dnmtpav,\munstar,\znstarwo)
  \end{align*}
  where $(\wntav, \dnmtav,\muntav,\zntav)
  \in\wnsetan\times\dnmset\times\munsetan\times\znset$, for each $t\ge
  1$, by convexity;
  and combining the resulting inequalities, we obtain
  \begin{align}\label{eq:relation-lag-stars}
    -\frac{\upperb(t,\munstar,\znstarwo)}{2} \le&\,
    \sum_{s=1}^{t}\phicc(\wns,\dnms,\muns,\zns)
    -t\phicc(\wnstarwo,\dnmstar,\munstar,\znstarwo) \notag
    \\
    \le&\, \frac{\upperb(t,\wnstarwo,\dnmstar)}{2}.
  \end{align}
  Choosing
  $(\wnp,\dnmp,\munp,\znp)=(\wntpav,\dnmtpav,\muntpav,\zntpav)$ in the
  bounds of Lemma~\ref{le:cummulative-concave-function-differences},
  multiplying each by $-1$ and combining them, we get
  \begin{align}\label{eq:relation-lag-av}
    &\,-\frac{\upperb(t,\wntpav,\dnmtpav)}{2} \le\Big(
    t\phicc(\wntpav,\dnmtpav,\muntpav,\zntpav) \notag
    \\
    &\,\qquad-\sum_{s=1}^{t}\phicc(\wns,\dnms,\muns,\zns)\Big)
    \le\frac{\upperb(t,\muntpav,\zntpav)}{2} .
  \end{align}
  The result now follows by summing~\eqref{eq:relation-lag-stars}
  and~\eqref{eq:relation-lag-av}.
\end{proof}

\subsection{Bounding the cumulative
  disagreement}\label{sec:cum-disagreement}

Given the dependence of the saddle-point evaluation error obtained in
Proposition~\ref{prop:approx-saddle-points} on the cumulative
disagreement of the estimates~$\dnmt$ and~$\znt$, here we bound their
disagreement over time.  We treat the subgradient terms as
perturbations in the dynamics~\eqref{eq:minmax-dyn-wnxnzn} and study
the input-to-state stability properties of the latter. This approach is well suited
for scenarios where the size of the subgradients can be uniformly
bounded.  Since the coupling in~\eqref{eq:minmax-dyn-wnxnzn}
with~$\wnt$ and $\munt$, as well as among the estimates $\dnmt$
and~$\znt$ themselves, takes place only through the subgradients, we
focus on the following pair of decoupled dynamics,
\begin{subequations}\label{eq:disagreement-minmax-dyn-wnxnzn}
  \begin{align}
    \dnmtphat &= \dnmt-\cstep\tlapb\dnmt
    +\pertdnmt\label{eq:disag-minmax-dyn-xn}
    \\
    \zntphat &=
    \znt-\cstep\tlapb\znt+\pertznt\label{eq:disag-minmax-dyn-zn}
    \\
    (\dnmtp,\zntp) &=
    \projec{\dnmset\times\znset}{\dnmtphat,\zntphat}\,,\notag
  \end{align}
\end{subequations}
where $\{\pertdnmt\}_{t\ge 1}\subset\realdimdn$, $\{\pertznt\}_{t\ge
  1}\subset\realdimzn$ are arbitrary sequences of disturbances,
and~$\mathcal{P}_{\dnmset\times\znset}$ is the orthogonal projection
onto~$\dnmset\times\znset$ {\color{black} as defined
  in~\eqref{eq:def-projection}.}

The next result characterizes the input-to-state stability properties
of~\eqref{eq:disagreement-minmax-dyn-wnxnzn} with respect to the
agreement space.  The analysis builds on the proof strategy in our
previous work~\cite[Prop. V.4]{DMN-JC:14-tnse}.  The main trick here
is to bound the projection residuals in terms of the disturbance. The
proof is presented in the Appendix.

\begin{proposition}\longthmtitle{Cumulative disagreement
    on~\eqref{eq:disagreement-minmax-dyn-wnxnzn} over
    jointly-connected weight-balanced
    digraphs}\label{prop:cum-disagreement-projected-subg-saddle}
  Let $\{\graph_s\}_{s\ge1}$ be a sequence of $B$-jointly connected,
  $\degn$-nondegenerate, weight-balanced digraphs.
  For $\degnt'\in(0,1)$, let
  \begin{align}\label{eq:degnt-condition-convergence-speed}
    \degnt\defin \min \big\{\, \degnt',\;
    (1-\degnt')\frac{\degn}{\doutmaxT}
    \,\big\} ,
  \end{align}
  where
  \begin{align*} 
    \doutmaxT \defin\max\setdefbig{\doutt(k)}{k\in\vertexset,\:
      t\in\integerspositive}.
  \end{align*}
  Then, for any choice of consensus stepsize such that
  \begin{align}\label{eq:eta-cstep-condition}
    \cstep \in\Big[\:\frac{\degnt}{\degn},\: \frac{1-\degnt}{
      \doutmaxT}\:\Big],
  \end{align}
  the dynamics~\eqref{eq:disag-minmax-dyn-xn} over~$\{\graph_t\}_{t\ge
    1}$ is input-to-state stable with respect to the nullspace of the
  matrix $\lapkbdouble$.  Specifically, for any
  $t\in\integerspositive$ and any $\{\persdnmt\}_{s=1}^{t-1} \subset
  \realdimdn$,
  \begin{align}\label{eq:iss-dnm}
    &\,\norm{\lapkb\dnmt} \le\tfrac{2^4\norm{\dnm_1}
    }{3^2}\Big(1-\frac{\degnt}{4
      N^2}\Big)^{\lceil\tfrac{t-1}{B}\rceil} +\consissu \!\max_{1\le
      s\le t-1}\norm{\persdnmt}\,,
  \end{align}
  where
  \begin{align}\label{eq:def-constant-disagreement}
    \consissu\defin \frac{2^5/3^2} {1- \big(1-\frac{\degnt}{4
        N^2}\big)^{1/B}}\,
  \end{align}
  and the cumulative disagreement satisfies
  \begin{align}\label{eq:cum-iss-dnm}
    \sum_{t=1}^{t'} \norm{\lapkb\dnmt} \le \consissu \Big(
    \tfrac{\norm{\dnm_1}}{2} +\sum_{t=1}^{t'-1}\norm{\pertdnmt}\Big).
  \end{align}
  Analogous bounds hold interchanging~$\dnmt$ with~$\znt$.
\end{proposition}

\subsection{Convergence of saddle-point subgradient dynamics with Laplacian
  averaging}\label{sec:proof-main-result-saddle-partial-agreement}

Here we characterize the convergence properties of the
dynamics~\eqref{eq:minmax-dyn-wnxnzn} using the developments above.
In informal terms, our main result states that, under a mild
connectivity assumption on the communication digraphs, a suitable
choice of decreasing stepsizes, and assuming that the agents'
estimates and the subgradient sets are uniformly bounded, the
saddle-point evaluation error under~\eqref{eq:minmax-dyn-wnxnzn}
decreases proportionally to~$\frac{1}{\sqrt{t}}$.  {\color{black} We
  select the learning rates according to the following scheme.
  \begin{assumption}\longthmtitle{\emph{Doubling Trick scheme} for the
      learning rates}\label{ass:Doubling-trick}
    The agents define a sequence of epochs numbered by
    $m=0,1,2,\dots$, and then use the constant value
    $\gradsteps=\tfrac{1}{\sqrt{2^m}}$ in each epoch~$m$, which has
    $2^m$ time steps $s=2^m,\dots,2^{m+1}-1$.
    Namely,
    \begin{alignat*}{2}
      \eta_1= &\;1\, , & \quad \eta_2=\eta_3=&\; 1/\sqrt{2}\,,
      \\
      \eta_4=\dots=\eta_7= &\;1/2 , & \qquad \eta_8=\dots=\eta_{15}
      =&\; 1/\sqrt{8}\,,
    \end{alignat*}
    and so on.  In general, 
    \begin{align*}
      \eta_{2^m}=\dots=\eta_{2^{m+1}-1}=1/\sqrt{2^m} .\eqoprocend
    \end{align*}
   %
  \end{assumption}
}

{\color{black} Note that the agents can compute the values in
  Assumption~\ref{ass:Doubling-trick} without communicating with each
  other.  Figure~\ref{fig:Doubling-trick} provides an illustration of
  this learning rate selection and compares it against constant and
  other sequences of stepsizes. Note that, unlike other choices
  commonly used in optimization~\cite{DPB-JNT:97,DPB-AN-AEO:03}, the
  Doubling Trick gives rise to a sequence of stepsizes that is not
  square summable.
  
  \begin{figure}[bth]
    \centering
    \includegraphics[width=.9\linewidth]{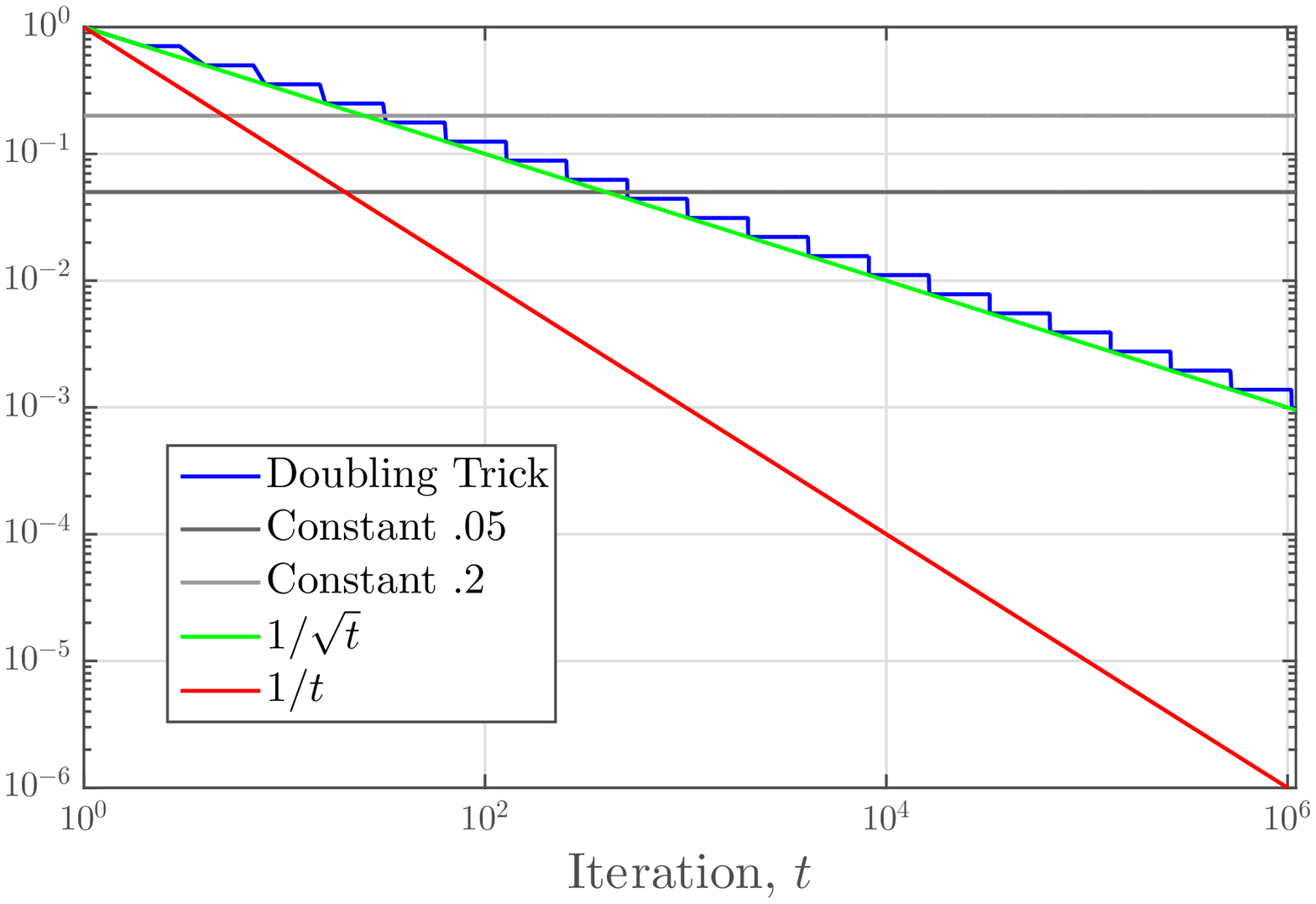}
    \caption{
    	    {\color{black}
    	Comparison of sequences of learning rates. We compare
      the sequence of learning rates resulting from
      the Doubling Trick in Assumption~\ref{ass:Doubling-trick}
      against a constant stepsize,  the sequence $\{
    1/\sqrt{t} \}_{t\ge 1}$, and the square-summable harmonic
    sequence $\{1/t\}_{t\ge 1}$. 
}
  }\label{fig:Doubling-trick}
\end{figure}
}

\begin{theorem}\longthmtitle{Convergence of the \emph{saddle-point
      dynamics with Laplacian
      averaging}~\eqref{eq:minmax-dyn-wnxnzn}}\label{th:convergence-general-saddle-point-w-d-mu-z}
  Let $\{(\wnt,\dnmt,\munt,\znt)\}_{t\ge 1}$ be generated
  by~\eqref{eq:minmax-dyn-wnxnzn} over a sequence
  $\{\graph_t\}_{t\ge1}$ of $B$-jointly connected,
  $\degn$-nondegenerate, weight-balanced digraphs satisfying
  $\sup_{t\ge 1} \sigmamax(\tlap)\le\lambdaup$ with~$\cstep$ selected
  as in~\eqref{eq:eta-cstep-condition}.  Assume
  \begin{align*}
    \norm{\wnt} \le\estboundwn,\;\;\norm{\dnmt}\le\estbounddnm,\;\;
    \norm{\munt}\le\estboundmun,\;\;\norm{\znt}\le\estboundzn,
  \end{align*}
  for all $t\in\integerspositive$ whenever the sequence of learning
  rates $\{\gradstept\}_{t\ge 1}\subset\realpositive$ is uniformly
  bounded.  Similarly, assume 
  \begin{align*}
    \norm{\subgwnt} \le\bsubgwn,\;\norm{\subgdnmt}\le\bsubgdnm,\;
    \norm{\subgmunt}\le\bsubgmun,\;\norm{\subgznt}\le\bsubgzn
  \end{align*}
  for all 
  $t\in\integerspositive$.
  Let the learning rates be chosen according to the \textit{Doubling
    Trick} in Assumption~\ref{ass:Doubling-trick}.
  %
  Then, for any saddle point~$(\wnstarwo, \dnmstar, \munstar,\znstarwo)$ 
  of~$\phicc$ on $\wnsetan\times\dnmset\times\munsetan\times\znset$ 
  with $\dnmstar=\dm^*\otimes\ones_N$ and 
  $\znstarwo=z^*\otimes\ones_N$ for some $(\dm^*,z^*)\in\dmset\times\zset$,
  which is assumed to exist, the following holds
  {\color{black}
  for the running time-averages:
  }
  \begin{align}\label{eq:general-convergence-bound}
    -\frac{ \cdoublingfac{\mun}{\zn}+\cdoublingfac{\wn}{\dnm}}{2\sqrt{t-1}}
    \le&\,
    \phicc(\wntav,\dnmtav,\zntav,\muntav)
    -\phicc(\wnstarwo,\dnmstar,\znstarwo,\munstar)
    \notag
    \\
    \le&\,
    \frac{\cdoublingfac{\wn}{\dnm}+\cdoublingfac{\mun}{\zn}}{2\sqrt{t-1}} ,
  \end{align}
  where $\cdoublingfac{\wn}{\dnm}\defin\frac{\sqrt{2}}{\sqrt{2}-1}
  \cdoubling{\wn}{\dnm}$ with
  \begin{align}\label{eq:def-cdoubling}
    \cdoubling{\wn}{\dnm}\defin &\,4(
    \estboundwn^2+\estbounddnm^2 )
    +6(\bsubgwn^2+\bsubgdnm^2) \notag
    \\
    &\,+\bsubgdnm (3+\cstep\lambdaup) \consissu 
    \big(\estbounddnm
    + 2\bsubgdnm\big)\,,
  \end{align}
  and $\cdoublingfac{\zn}{\mun}$ is analogously defined.
\end{theorem}
\begin{proof}
  We divide the proof in two steps. In step~(i), we use the general
  bound of Proposition~\ref{prop:approx-saddle-points} making a choice
  of constant learning rates over a fixed time horizon~$t'$.  In
  step~(ii), we use multiple times this bound together with the
  Doubling Trick to produce the implementation procedure in the
  statement.
  In~(i), to further bound~\eqref{eq:general-bound-lower-upper}, we
  choose $\gradstept=\eta'$ for all $s\in\until{t'}$ in
  both~$\upperb(t',\wnstarwo,\dnmstar)$ and $\upperb(t',
  \wntspav,\dnmtspav)$. By doing this, we make zero the first two
  lines in~\eqref{eq:cummulative-difference-dnmt-xn}, and then we
  upper-bound the remaining terms using the bounds on the estimates
  and the subgradients.  The resulting inequality also holds replacing
  $(\wntspav,\dnmtspav)$ by $(\wnstarwo,\dnmstar)$,
  \begin{align}\label{eq:upper-term}
    \upperb(t', \wntspav,\dnmtspav)
    \le
    &\,\tfrac{2}{\eta'} 
    \big(\norm{\wn_{1}}^2+\estboundwn^2
    +\norm{\dnm_{1}}^2+\estbounddnm^2\big)  
    \nonumber   
    \\
    &\,+6(\bsubgwn^2+\bsubgdnm^2) \eta' t'
    \nonumber 
    \\
    +2(2+\cstep\lambdaup) \bsubgdnm \sum_{s=1}^{t'} &\,\norm{\lapkbk\dnms}
    +2\norm{\lapkbk\dnmtspav} \bsubgdnm t'\,.
  \end{align}
  Regarding the bound for $\upperb(t',\wnstarwo,\dnmstar)$, we just
  note that $\norm{\lapkbk\dnmstar}=0$, whereas for $\upperb(t',
  \wntspav,\dnmtspav)$, we note that, by the triangular inequality, we
  have
  \begin{align*}
    \norm{\lapkbk\dnmtspav}
    =&\,\frac{1}{t'}\norm{\lapkbk\big(\sum_{s=1}^{t'} \dnms\big)} \le
    \frac{1}{t'}\sum_{s=1}^{t'}\norm{\lapkbk \dnms} .
  \end{align*}
  That is, we get 
  \begin{align}\label{eq:upper-term-u-disagreement}
    &\,\upperb(t',\wnstarwo,\dnmstar)\le\upperb(t', \wntspav,\dnmtspav)
    \notag
    \\
    \le
    &\,\tfrac{2}{\eta'}
    \big(\norm{\wn_{1}}^2+\estboundwn^2
    +\norm{\dnm_{1}}^2+\estbounddnm^2\big)  
    \nonumber   
    +6(\bsubgwn^2+\bsubgdnm^2) \eta' t'
    \nonumber 
    \\
    &\, +2\bsubgdnm (3+\cstep\lambdaup) 
    \sum_{s=1}^{t'}\norm{\lapkbk\dnms} .
  \end{align}
  We now further bound $\sum_{s=1}^{t'} \norm{\lapkbk\dnms}$
  in~\eqref{eq:cum-iss-dnm} noting that
  $\norm{\pertdnmt}=\norm{\gradstept\subgdnmt}\le\gradstept
  \bsubgdnm=\eta'\bsubgdnm$, to obtain
  \begin{align*} 
    \sum_{s=1}^{t'} \norm{\lapkbk\dnms} \le &\,\consissu \Big(
    \tfrac{\norm{\dnm_1}}{2} +\sum_{s=1}^{t'-1}\eta'\bsubgdnm\Big)
    \\
    \le &\, \consissu \big( \tfrac{\norm{\dnm_1}}{2}
    +t'\eta'\bsubgdnm\big) .
  \end{align*}
  Substituting this bound in~\eqref{eq:upper-term-u-disagreement},
  taking $\eta'=\frac{1}{\sqrt{t'}}$ and noting that $1\le\sqrt{t'}$,
  we get
  \begin{align}\label{eq:upper-term-u-saddle-learning-rates}
    \upperb(t', \wntspav,\dnmtspav)
    \le
    \alpha'\sqrt{t'} ,
  \end{align}
  where
  \begin{align*}
    \alpha'\defin &\,2(\norm{\wn_{1}}^2+\norm{\dnm_{1}}^2
    +\estboundwn^2+\estbounddnm^2 )
    +6(\bsubgwn^2+\bsubgdnm^2)
    \\
    &\,+2\bsubgdnm (3+\cstep\lambdaup) \consissu 
    \big(\tfrac{\norm{\dnm_1}}{2}
    + \bsubgdnm\big) .
  \end{align*}
  This bound is of the type $\upperb(t',
  \wntspav,\dnmtspav)\le\alpha'\sqrt{t'}$, where $\alpha'$ depends on
  the initial conditions. This leads to step~(ii).  According to the
  Doubling Trick, for $m=0,1,\dots \lceil\log_2 t\rceil $, the
  dynamics is executed in each epoch of $t'=2^m$ time steps
  $t=2^m,\dots,2^{m+1}-1$, where at the beginning of each epoch the
  initial conditions are the final values in the previous epoch. The
  bound for $\upperb(t', \wntspav,\dnmtspav)$ in each epoch is
  $\alpha'\sqrt{t'}=\alpha_m\sqrt{2^m}$, where $\alpha_m$ is the
  multiplicative constant
  in~\eqref{eq:upper-term-u-saddle-learning-rates} that depends on the
  initial conditions in the corresponding epoch.
  %
  Using the assumption that the estimates are bounded, i.e.,
  $\alpha_m\le \cdoubling{\wn}{\dnm}$, we deduce that the bound in
  each epoch is $\cdoubling{\wn}{\dnm}\sqrt{2^m}$. By the Doubling
  Trick,
  \begin{align*}
    \sum_{m=0}^{\lceil\log_2 t\rceil} \sqrt{2^m}
    =\tfrac{1-\sqrt{2}^{\lceil\log_2 t\rceil+1}}{1-\sqrt{2}}
    \le
    \tfrac{1-\sqrt{2t}}{1-\sqrt{2}}
    \le \tfrac{\sqrt{2}}{\sqrt{2}-1}\sqrt{t} ,
  \end{align*}
  we conclude that
  \begin{align*}
    \upperb(t,\wnstarwo,\dnmstar)\le \upperb(t, \wntpav,\dnmtpav)\le
    \tfrac{\sqrt{2}}{\sqrt{2}-1} \cdoubling{\wn}{\dnm}\sqrt{t} .
  \end{align*}
  Similarly,
  \begin{align*}
    -\upperb(t,\munstar,\znstarwo)\ge -\upperb(t, \muntpav,\zntpav)\ge
    -\tfrac{\sqrt{2}}{\sqrt{2}-1} \cdoubling{\mun}{\zn}\sqrt{t} .
  \end{align*}
  The desired pair of inequalities follows substituting these bounds
  in~\eqref{eq:general-bound-lower-upper} and dividing by~$2t$.
\end{proof}

{\color{black} In the statement of
  Theorem~\ref{th:convergence-general-saddle-point-w-d-mu-z}, the
  constant $\consissu$ appearing in~\eqref{eq:def-cdoubling} encodes
  the dependence on the network properties.  The running time-averages
  can be updated sequentially as $\wntpav\defin
  \frac{t-1}{t}\wntav+\frac{1}{t}\wnt$ without extra memory.  Note
  also that we assume feasibility of the problem because this property
  does not follow from the behavior of the algorithm.

\begin{remark}\longthmtitle{Boundedness of
    estimates}\label{re:boundedness-estimates}
  {\rm The statement of
    Theorem~\ref{th:convergence-general-saddle-point-w-d-mu-z}
    requires the subgradients and the estimates produced by the
    dynamics to be bounded. In the literature of distributed (sub-)
    gradient methods, it is fairly common to assume the boundedness of
    the subgradient sets relying on their continuous dependence on the
    arguments, which in turn are assumed to belong to a compact
    domain.  Our assumption on the boundedness of the estimates,
    however, concerns a saddle-point subgradient dynamics for general
    convex-concave functions, and its consequences vary depending on
    the application. We come back to this point and discuss the
    treatment of dual variables for distributed constrained
    optimization in
    Section~\ref{sec:distributed-bound-optimal-dual-set}. \oprocend }
\end{remark}
}

\section{Applications to distributed constrained convex
  optimization}\label{sec:application-constrained-optimization}

In this section we particularize our convergence result in
Theorem~\ref{th:convergence-general-saddle-point-w-d-mu-z} to the case
of convex-concave functions arising from the Lagrangian of the
constrained
optimization~\eqref{eq:problem-linear-inequality-constraints}
discussed in Section~\ref{sec-linear-semidefinite-constraints-intro}.
The Lagrangian formulation with explicit agreement
constraints~\eqref{eq:separable-constraints-c} matches the general
saddle-point problem~\eqref{eq:min-max-convex-concave-function} for
the convex-concave
function~$\map{\phicc}{\prodwseti\times\dnmset\times
  (\realnonnegative^{m})^N}{\real}$ defined by
\begin{align}\label{eq:lagrangian-global-constraint}
  \phicc(\wn,\dnm,\zn)= \sum_{i=1}^N \big( f^i(w^i,\dm^i) + {z^i}\tp
  g^{i}(w^i,\dm^i) \big) .
  %
\end{align}
Here the arguments of the convex part are, on the one hand, the local
primal variables across the network, $\wn=(w^1,\dots,w^N)$ (not
subject to agreement), and, on the other hand, the copies across the
network of the global decision vector, $\dnm=(\dm^1,\dots,\dm^N)$
(subject to agreement).  The arguments of the concave part
are the network estimates of the Lagrange multiplier,
$\zn=(z^1,\dots,z^N)$ (subject to agreement).
Note that this convex-concave function is the associated Lagrangian
for~\eqref{eq:problem-linear-inequality-constraints} \emph{only} under
the agreement on the global decision vector and on the Lagrange
multiplier associated to the global constraint, i.e.,
\begin{align}\label{eq:lagrangian-global-constraint-agreement}
  \lag(\wn,\dm,z)= \phicc(\wn,\dm\otimes\ones_N,z\otimes\ones_N).
\end{align}
In this case, the \emph{saddle-point dynamics with Laplacian
  averaging}~\eqref{eq:minmax-dyn-wnxnzn} takes the following form:
the updates of each agent $i\in\until{N}$ are as follows,
\begin{subequations}\label{eq:minmax-dyn-constrained-agent-updates}
  \begin{align}
    \hat{w}^i_{t+1}=&\,w^i_t-\gradstept
    (\subdarg{f^i,w^i_t}+\subdarg{g^i,w^i_t}\tp\, {z^i}) ,
    \label{eq:minmax-dyn-constrained-agent-w}
    \\
    \hat{\dm}^i_{t+1}=&\,\dm^i_t+\cstep\sum_{j=1}^N \adj_{ij,t} (\dm^j_t-\dm^i_t)
    \notag
    \\
    &\,\quad-\gradstept (\subdarg{f^i,\dm^i_t}+\subdarg{g^i,\dm^i_t}\tp\, {z^i}) ,
    \label{eq:minmax-dyn-constrained-agent-dm}
    \\
    \hat{z}^i_{t+1}=&\,z^i_t+\cstep\sum_{j=1}^N\adj_{ij,t} (z^j_t-z^i_t)
    +\gradstept g^i(w^i_t) ,
    \label{eq:minmax-dyn-constrained-agent-z}
    \\
    \begin{bmatrix}
      w^i_{t+1}
      \\
      \dm^i_{t+1}
      \\
      z^i_{t+1}
    \end{bmatrix}
    =&\,
    \begin{bmatrix}
      \projec{\wseti}{\hat{w}^i_{t+1}}
      \\
      \projec{\dmset}{\hat{\dm}^i_{t+1}}
      \\
      \projec{\realnonnegative^{m}\cap\ballc{0}{r}}{\hat{z}^i_{t+1}}
    \end{bmatrix} ,
    \label{eq:minmax-dyn-constrained-agent-projection}
  \end{align}
\end{subequations}
%
%
where the vectors $\subdarg{f^i,w^i_t}\in\real^{n_i}$ and
$\subdarg{f^i,\dm^i_t}\in\realdimd$ are subgradients of $f^i$ with
respect to the first and second arguments, respectively, at the point
$(w^i_t,\dm^i_t)$, i.e.,
\begin{align}\label{eq:corollary-subgradient-choices}
  \subdarg{f^i,w^i_t}\in\partial_{w^i}{f^i}(w^i_t,\dm^i_t), \qquad
  \subdarg{f^i,\dm^i_t}\in\partial_{\dm}{f^i}(w^i_t,\dm^i_t) ,
\end{align}
and the matrices $\subdarg{g^i,w^i}\in\real^{m\times {n_i}}$ and
$\subdarg{g^i,\dm}\in\real^{m\times \dimdn}$ contain in the $l$th row
an element of the subgradient sets
$\partial_{w^i}{g^i_l}(w^i_t,\dm^i_t)$
and~$\partial_{\dm}{g^i_l}(w^i_t,\dm^i_t)$, respectively.
{\color{black} (Note that these matrices correspond, in the
  differentiable case, to the Jacobian block-matrices of the vector
  function
  $\map{g^i}{\real^{n_i}\times\realdimd}{\real^m}$.) We refer to this
  strategy as the \emph{Consensus-based Saddle-Point (Sub-) Gradient
    (C-SP-SG) algorithm} and present it in pseudo-code format in
  Algorithm~\ref{alg:saddle-point-constrained-optimization}.}

{\color{black} Note that the orthogonal projection of the estimates of the multipliers
in~\eqref{eq:minmax-dyn-constrained-agent-projection} is unique.  }
The radius~$r$ employed in its definition is a
design parameter that is either set \emph{a priori} or determined by
the agents.  We discuss this point in detail below in
Section~\ref{sec:distributed-bound-optimal-dual-set}.
%
%


\begin{algorithm} 
  {\color{black}
  \vspace{4pt}
 \SetKwInput{KwInit}{Initialization}
  \KwData{Agents' data for
    Problem~\eqref{eq:problem-linear-inequality-constraints}:
    $\{f^i,g^i,\wseti\}_{i=1}^N$, $\dmset$
    \\ 
    \quad Agents' adjacency values $\{\Adj_t\}_{t\ge 1}$
    \\
    \quad Consensus stepsize~$\cstep$ as
    in~\eqref{eq:eta-cstep-condition}
    \\
    \quad Learning rates $\{\gradstep_t\}_{t\ge 1}$ as in
    Assumption~\ref{ass:Doubling-trick}
    \\
    \quad Radius $r$ s.t. $\ballc{0}{r}$ contains optimal dual set
    for~\eqref{eq:problem-linear-inequality-constraints}
    \\
    \quad Number of iterations $T$, indep. of rest of parameters
    \\
  }
  \KwResult{Agent $i$ outputs $ \wiTav, \dmiTav, \ziTav$}
  \KwInit{Agent $i$ sets $w^i_1\in\real^{n_i}$, $\dm^i_1\in\realdimd$,
    $z^i_1\in\realnonnegative^{m}$, $(w^i_1)^{\text{av}}=w^i_1$,
    $(\dm^i_1)^{\text{av}}=\dm^i_1$, $(z^i_1)^{\text{av}}=z^i_1$}
  \For{$t\in\{2,\dots,T-1 \}$}{ \For{$i \in \until{N}$}{ \vspace{3pt}
      Agent $i$ selects (sub-) gradients as
      in~\eqref{eq:corollary-subgradient-choices}
      \\
      \vspace{3pt} Agent $i$ updates $(\wit, \dmit, \zit)$ as
      in~\eqref{eq:minmax-dyn-constrained-agent-updates}
      \\
      \vspace{3pt}
      Agent $i$ updates $\witpav=\frac{t-1}{t}\witav+\frac{1}{t}\wit$,\\
      \vspace{2pt}
      \qquad\qquad$\dmitpav=\frac{t-1}{t}\dmitav+\frac{1}{t}\dmit$ , \\
      \vspace{2pt}
      \qquad\qquad$\zitpav=\frac{t-1}{t}\zitav+\frac{1}{t}\zit$
      \vspace{2pt}
    }
  }
  \caption{C-SP-SG
    algorithm}\label{alg:saddle-point-constrained-optimization}
  }
\end{algorithm}

The characterization of the saddle-point evaluation error
under~\eqref{eq:minmax-dyn-constrained-agent-updates} is a direct
consequence of
Theorem~\ref{th:convergence-general-saddle-point-w-d-mu-z}.

\begin{corollary}\longthmtitle{Convergence of the C-SP-SG
    algorithm}\label{cor:app-constrained-opt}
  For each $i\in\until{N}$, let the sequence $\{(w^i_t, \dm^i_t,
  z^i_t)\}_{t\ge 1}$ be generated by the coordination
  algorithm~\eqref{eq:minmax-dyn-constrained-agent-updates}, over a
  sequence of graphs $\{\graph_t\}_{t\ge1}$ satisfying the same
  hypotheses as
  Theorem~\ref{th:convergence-general-saddle-point-w-d-mu-z}.
  Assume that the sets $\dmset$ and~$\wseti$ are compact (besides
  being convex), and the radius $r$ is such that~$\ballc{0}{r}$
  contains the optimal dual set of the constrained
  optimization~\eqref{eq:problem-linear-inequality-constraints}.
  %
  %
    %
  Assume also that 
  the subgradient sets are
  bounded, in $\wseti\times\dmset$, as follows,
  \begin{align*}
    \partial_{w^i}{f^i}\subseteq\ballc{0}{H_{f,w}}&\, ,\;
    \partial_{\dm}{f^i}\subseteq\ballc{0}{H_{f,\dm}} ,\;
    \\
    \partial_{w^i}{g^i_l}\subseteq\ballc{0}{H_{g,w}}&\, ,\;
    \partial_{\dm}{g^i_l}\subseteq \ballc{0}{H_{g,\dm}} ,
  \end{align*}
  for all $l\in\until{m}$.
  Let~$(\wnstarwo,\dm^*,z^*)$ be any saddle point of the
  Lagrangian~$\lag$ defined
  in~\eqref{eq:lagrangian-global-constraint-agreement} on the set
  $\prodwseti\times\dmset\times\real^m$. (The existence of such
  saddle-point implies that strong duality is attained.)  Then, under
  Assumption~\ref{ass:Doubling-trick} for the learning rates,
  the
  saddle-point evaluation error~\eqref{eq:general-convergence-bound}
  holds for the running time-averages:
  \begin{align}\label{eq:constrained-optimization-convergence-bound}
    -\frac{ \cdoublingfac{\mun}{\zn}+\cdoublingfac{\wn}{\dnm}}{2\sqrt{t-1}}
    \le&\,
    \phicc(\wntav,\dnmtav,\zntav)
    -\lag(\wnstarwo,\dm^*,z^*)
    \notag
    \\
    \le&\,
    \frac{\cdoublingfac{\wn}{\dnm}+\cdoublingfac{\mun}{\zn}}{2\sqrt{t-1}} ,
  \end{align}
  for~$\cdoublingfac{\wn}{\dnm}$ and $\cdoublingfac{\mun}{\zn}$ as
  in~\eqref{eq:def-cdoubling}, with
  \begin{align*}
    \estboundmun=&\,\bsubgmun=0 ,
    \qquad
        \estboundzn= \sqrt{N} r ,
    \\
    \estboundwn=&\,(\sum_{i=1}^N\diam(\wseti)^2)^{1/2} ,
    \qquad
    \estbounddnm=\sqrt{N}\diam(\dmset) ,
    \\
    \bsubgwn^2=&\,
    N (H_{f, w}+r \sqrt{m}H_{g, w})^2 ,
    \quad\!\!
    \bsubgzn^2=\sum_{i=1}^N (\sup_{w^i\in\wseti} g_i(w^i))^2 ,
    \\
    \bsubgdnm^2=&\,N (H_{f,\dm}+r \sqrt{m}H_{g,\dm})^2 ,
  \end{align*}
  {\color{black} where $\diam(\cdot)$ refers to the diameter of the
    sets.  }
\end{corollary}

The proof of this result follows by noting that the hypotheses of
Theorem~\ref{th:convergence-general-saddle-point-w-d-mu-z} are
automatically satisfied. The only point to observe is that all the
saddle points of the Lagrangian $\lag$ defined
in~\eqref{eq:lagrangian-global-constraint-agreement} on the set
$\prodwseti\times\dmset\times\realnonnegative^{m}$, are also contained
in~$\prodwseti\times\dmset\times\ballc{0}{r}$.  {\color{black} Note
  also that we assume feasibility of the problem because this property
  does not follow from the behavior of the algorithm.}

{\color{black}
  \begin{remark}\longthmtitle{Time, memory, computation, and
      communication complexity of the C-SP-SG
      algorithm}\label{re:complexity-particular-dynamics}
    {\rm We discuss here the complexities associated with the
      execution of the C-SP-SG algorithm:
      \begin{itemize}
      \item \textbf{Time complexity}: According to
        Corollary~\ref{cor:app-constrained-opt}, the saddle-point
        evaluation error is smaller than $\eps$ if
        $\frac{\cdoublingfac{\wn}{\dnm}+\cdoublingfac{\mun}{\zn}}{2\sqrt{t}}\le
        \eps$. This provides a lower bound
        \begin{align*}
          t\ge \big(\frac{\cdoublingfac{\wn}{\dnm} +
            \cdoublingfac{\mun}{\zn}}{2\eps}\big)^2 ,
        \end{align*}
        on the number of required iterations.
      \item \textbf{Memory complexity}: Each agent~$i$ maintains the
        current updates $(w^i_{t}, \dm^i_{t},
        z^i_{t})\in\real^{n_i}\times\realdimd\times\real^{m}$, and the
        corresponding current running time-averages $(\witav, \dmitav,
        \zitav)$ with the same dimensions.
      \item \textbf{Computation complexity}: Each agent~$i$ makes a
        choice/evaluation of subgradients, at each iteration, from the
        subdifferentials~$\partial_{w^i}{f^i}\subseteq\real^{n_i}$,
        $\partial_{\dm}{f^i}\subseteq\realdimd$,
        $\partial_{w^i}{g^i_l}\subseteq\real^{n_i}$ ,
        $\partial_{\dm}{g^i_l}\subseteq\realdimd$, the latter for
        $l\in\until{m}$.
        Each agent also projects its estimates on the
        set~$\wseti\times\dmset\times\realnonnegative^{m}\cap\ballc{0}{r}$. The
        complexity of this computation depends on the sets $\wseti$
        and $\dmset$.
        
      \item \textbf{Communication complexity:} Each agent $i$ shares
        with its neighbors at each iteration a vector
        in~$\realdimd\times\real^{m}$. With the information received,
        the agent updates the global decision variable $\dm^i_{t}$
        in~\eqref{eq:minmax-dyn-constrained-agent-dm} and the Lagrange
        multiplier $z^i_{t}$
        in~\eqref{eq:minmax-dyn-constrained-agent-z}.  (Note that the
        variable $\dm^i_{t}$ needs to be maintained and communicated
        only if the optimization
        problem~\eqref{eq:problem-linear-inequality-constraints} has a
        global decision variable.)       \oprocend
      \end{itemize}
    }
  \end{remark}
}

\subsection{Distributed strategy to bound the optimal dual
  set}\label{sec:distributed-bound-optimal-dual-set}

The motivation for the design choice of \emph{truncating} the
projection of the dual variables onto a bounded set
in~\eqref{eq:minmax-dyn-constrained-agent-projection} is the
following. The subgradients of~$\phicc$ with respect to the primal
variables are \emph{linear} in the dual variables.  To guarantee the
boundedness of the subgradients of~$\phicc$ and of the dual variables,
required by the application of
Theorem~\ref{th:convergence-general-saddle-point-w-d-mu-z}, one can
introduce a projection step onto a compact set that preserves the
optimal dual set, a technique that has been used
in~\cite{AN-AO:09-jota,AN-AO:10-siam,THC-AN-AS:14}.
These works select the bound for the projection \emph{a priori},
whereas~\cite{MZ-SM:12} proposes a distributed algorithm to compute
a bound preserving the optimal dual set,
\emph{for} the case of a global inequality
constraint \emph{known to all the agents}. Here, we deal with a
complementary case, where the constraint is a sum of functions, each
known to the corresponding agents, that couple the local decision
vectors across the network.  For this case, we next describe how the
agents can compute, in a distributed way,
a radius $r\in\realpositive$ such that the ball $\ballc{0}{r}$
contains the optimal dual set for the constrained
optimization~\eqref{eq:problem-linear-inequality-constraints}.
{\color{black} A radius with such property is not unique, and estimates
  with varying degree of conservativeness are possible.  }


In our model,
each agent $i$ has only access to the set~$\wseti$ and the
functions~$f^i$ and $g^i$. In turn, we make the important assumption
that there are no variables subject to agreement, i.e.,
$f^i(w^i,\dm)=f^i(w^i)$ and $g^i(w^i,\dm)=g^i(w^i)$ for all
$i\in\until{N}$, and we leave for future work the generalization to
the case where agreement variables are present.  Consider then the
following problem,
\begin{align}\label{eq:problem-constraints-wo-dnm}
  \min_{w^i\in\wseti,\,\forall i }&\, \sum_{i=1}^N f^i(w^i) \notag
  \\
  \rm{s.t.}&\, g^1(w^1)+\dots+g^N(w^N)\le 0
\end{align} {\color{black} where each $\wseti$ is compact as in
  Corollary~\ref{cor:app-constrained-opt}.  }
We first propose a bound on the optimal dual set and then describe a
distributed strategy that allows the agents to compute it. Let
$(\tilde{w}^1,\dots, \tilde{w}^N)\in\prodwsetiwo$ be a vector
satisfying the \emph{Strong Slater
  condition}~\cite[Sec. 7.2.3]{JBHU-CL:93}, called \emph{Slater
  vector}, and define
\begin{align}\label{eq:definition-gamma}
\gamma\defin\min_{l\in\until{m}} -\sum_{i=1}^N g^i_l(\tilde{w}^i) ,
\end{align}
which is positive by construction.
According to~\cite[Lemma 1]{AN-AO:10-siam} (which we amend imposing
that the Slater vector belongs to the abstract constraint
set~$\prodwseti$), we get that the optimal dual
set~$\dualoptset\subseteq\realnonnegative^m$ associated to the
constraint $g^1(w^1)+\dots+g^N(w^N)\le 0$ is bounded as follows,
\begin{align}\label{eq:bound-Asu-multipliers}
  \max_{z^*\in \dualoptset}\norm{z^*}\le \frac{1}{\gamma}
  \big(\sum_{i=1}^N f^i(\tilde{w}^i)-q(\bar{z})\big) ,
\end{align}
for any $\bar{z}\in\realnonnegative^{m}$, where
$\map{q}{\realnonnegative^{m}}{\real}$ is the dual function associated
to the optimization~\eqref{eq:problem-constraints-wo-dnm},
\begin{align}
  q(z)&= \inf_{w^i\in\wseti,\,\forall i}\;\;\lag(\wn,z)\notag
  \\
  &=\inf_{w^i\in\wseti,\,\forall i}\;\; \sum_{i=1}^N \big( f^i(w^i) +
  z\tp g^{i}(w^i) \big) =:\sum_{i=1}^N q^i(z) .\label{eq:def-q-dual}
\end{align}
{\color{black} Note that the right hand side
  in~\eqref{eq:bound-Asu-multipliers} is nonnegative by weak duality,
  and that $q(\bar{z})$ does not coincide with $-\infty$ for any
  $\bar{z}\in\realnonnegative^{m}$ because each set~$\wseti$ is
  compact.  }
%
%
{\color{black} With this notation,
\begin{align*}
  \sum_{i=1}^N f^i(\tilde{w}^i)-q(\bar{z})\le
  N\Big(\max_{j\in\until{N}} f^j(\tilde{w}^j)-\!\!\min_{j\in\until{N}}
  q^j(\bar{z})\Big).
\end{align*}
Using this bound in~\eqref{eq:bound-Asu-multipliers}, we conclude that
$\dualoptset\subseteq\zsetcom$, with
\begin{align}\label{eq:compact-set-multipliers}
  \zsetcom\defin
  &\,\realnonnegative^{m}\cap\ballcbig{0}{\frac{N}{\gamma}
    \Big(\max_{j\in\until{N}} f^j(\tilde{w}^j)-\!\!\min_{j\in\until{N}}
    q^j(\bar{z})\Big) } .
\end{align}
} Now we briefly describe the distributed strategy that the agents can
use to bound the set~$\zsetcom$.  The algorithm can be divided in
three stages:
\begin{enumerate}
\item[(i.a)] Each agent finds the corresponding component~$\tilde{w}^i$
  of a Slater vector.
\end{enumerate}
%
For instance, if $\wseti$ is \emph{compact} (as is the case in
Corollary~\ref{cor:app-constrained-opt}), agent~$i$ can compute
\begin{align*}
  \tilde{w}^i\in\argmin_{w^i\in\wseti} g^i_l(w^i) .
\end{align*}
The resulting vector $(\tilde{w}^1,\dots, \tilde{w}^N)$ is a Slater
vector, i.e., it belongs to the set
\begin{align*}
  \{(w^1,\dots, w^N)\in&\,\prodwsetiwo \::
  \\&\,g^1(w^1)+\dots+g^N(w^N)< 0\} ,
\end{align*}
which is nonempty by the Strong Slater condition.
%
%
{\color{black}
  \begin{enumerate}
  \item[(i.b)] Similarly, the agents compute the corresponding
    component $q^i(\bar{z})$ defined in~\eqref{eq:def-q-dual}. The
    common value~$\bar{z}\in\realnonnegative^{m}$ does not depend on
    the problem data and can be~$0$ or any other value agreed upon by
    the agents beforehand.
  \end{enumerate}
}
%
%
%
\begin{enumerate}
\item[(ii)] The agents find a lower bound for~$\gamma$
  in~\eqref{eq:definition-gamma} in two stages: first they use a
  distributed consensus algorithm and at the same time they estimate
  the fraction of agents that have a positive estimate. Second, when
  each agent is convinced that every other agent has a positive
  approximation, given by a precise termination condition that is
  satisfied in finite time, they broadcast their estimates to their
  neighbors to agree on the minimum value across the network.
\end{enumerate}
Formally, each agent sets $y^i(0)\defin g^i(\tilde{w}^i)\in\real^m$
and $s_i(0)\defin\sign(y^i(0))$, and executes the following iterations
\begin{subequations}
  \begin{align}
    y^i(k+1)=&\,y^i(k)+\cstep\sum_{j=1}^N\adj_{ij,t} (y^j(k)-y^i(k))
    ,\label{eq:standard-consensus}
    \\
    s_i(k+1)=&\,s_i(k) +\cstep\sum_{j=1}^N\adj_{ij,t}
    \big(\sign(y^j(k))\notag \\
    &\,\qquad\qquad\qquad\qquad-\sign(y^i(k))\big) ,
  \end{align}
\end{subequations}
until an iteration $k^*_i$ such that $Ns_i(k^*_i)\le -(N-1)$; see
Lemma~\ref{le:termination-condition} below for the justification of
this termination condition.  Then, agent $i$ re-initializes
$y^i(0)=y^i(k^*)$ and iterates
\begin{align}\label{eq:max-agreement}
  y^i(k+1)=\min \setdef{y^j(k)}{j\in\Nout(i)\cup \{i\}}    
\end{align}
(where agent $i$ does not need to know if a neighbor has
re-initialized).  The agents reach agreement about
$\min_{i\in\until{n}} y^i(0)=\min_{i\in\until{n}} y^i(k^*)$ in a
number of iterations no greater than $(N-1)B$ counted after
$k^{**}\defin\max_{j\in\until{N}} k^*_j$ (which can be computed if
each agent broadcasts once $k_i^*$). Therefore, the agents obtain the
same lower bounds
\begin{align*}
  \hat{y}\defin&\, N y^i(k^{**})\le \sum_{i=1}^N g^i(\tilde{w}^i) ,
  \\
  \lowergamma\defin&\,\min_{l\in\until{m}} -\hat{y}_l\le \gamma ,
\end{align*}
where the first lower bound is coordinate-wise.
\begin{enumerate}
\item[(iii)] The agents exactly agree on $\max_{j\in\until{N}}
  f^j(\tilde{w}^j)$ and {\color{black}
    $\min_{j\in\until{N}}q^j(\bar{z})$ } using the finite-time
  algorithm analogous to~\eqref{eq:max-agreement}.
\end{enumerate}
In summary, the agents obtain the same upper bound
\begin{align*}
  r\defin\frac{N}{\lowergamma}\Big(\max_{j\in\until{N}}
  f^j(\tilde{w}^j)-\!\min_{j\in\until{N}} q^j(\bar{z})\Big),
\end{align*}
which, according to~\eqref{eq:compact-set-multipliers}, bounds the
optimal dual set for the constrained
optimization~\eqref{eq:problem-constraints-wo-dnm},
\begin{align*}
  \dualoptset\subseteq\zsetcom\subseteq\ballc{0}{r} .
\end{align*}
To conclude, we justify the termination condition of step~(ii).

\begin{lemma}\longthmtitle{Termination condition of step
    (ii)}\label{le:termination-condition}
  If each agent knows the size of the network $N$, then under the same
  assumptions on the communication graphs and the parameter $\cstep$
  as in~Theorem~\ref{th:convergence-general-saddle-point-w-d-mu-z},
  the termination time $k^*_i$ is {\color{black} finite.}
\end{lemma}
\begin{proof}
  Note that $y^i(0)$ is not guaranteed to be negative but, by
  construction of each $\{g^i(\tilde{w}^i)\}_{i=1}^N$ in step (i), it
  holds that the convergence point for~\eqref{eq:standard-consensus}
  is
  \begin{align}\label{eq:point-convergence-consensus-appendix}
    \frac{1}{N}\sum_{i=1}^N y^i(0)=\frac{1}{N}\sum_{i=1}^N
    g^i(\tilde{w}^i)<0 .
  \end{align}
  This, together with the fact that Laplacian averaging preserves the
  convex hull of the initial conditions, it follows (inductively) that
  $s_i$ decreases monotonically to $-1$. Thanks to the exponential
  convergence of~\eqref{eq:standard-consensus} to the
  point~\eqref{eq:point-convergence-consensus-appendix}, it follows
  that there exists a finite time $k^*_i\in\integerspositive$ such
  that $Ns_i(k^*_i)\le -(N-1)$. This termination time is determined by
  the constant $B$ of joint connectivity and the constant $\degn$ of
  nondegeneracy of the adjacency matrices.
\end{proof}
{\color{black} The complexity of the entire procedure corresponds to
  \begin{itemize}
  \item each agent computing the minimum of two convex functions;
  \item executing Laplacian average consensus until the agents'
    estimates fall within a centered interval around the average of
    the initial conditions; and
  \item running two agreement protocols on the minimum of quantities
    computed by the agents.
  \end{itemize}
}
%

\section{Simulation example}

{\color{black} Here we simulate\footnote{The Matlab code is available
    at
    \url{https://github.com/DavidMateosNunez/Consensus-based-Saddle-Point-Subgradient-Algorithm.git}.}
  the performance of the Consensus-based Saddle-Point (Sub-) Gradient
  algorithm
  (cf. Algorithm~\ref{alg:saddle-point-constrained-optimization}) in a
  network of $N=50$ agents whose communication topology is given by a
  fixed connected small world graph~\cite{DJW-SHS:98} with maximum
  degree~$\doutmaxT=4$.  Under this coordination strategy, the~$50$
  agents
solve collaboratively the following instance of
problem~\eqref{eq:problem-linear-inequality-constraints} with
nonlinear convex constraints:
\begin{align}\label{eq:problem-simulation-log-constraints}
  \min_{w_i\in[0, 1]}&\,
  \sum_{i=1}^{50} c_i w_i \notag
  \\
  \rm{s.t.}&\, \sum_{i=1}^{50} -d_i \log(1 + w_i)\le -b .
\end{align}
Problems with constraints of this form arise, for instance, in
wireless networks to ensure quality-of-service.  For each
$i\in\until{50}$, the constants $c_i$, $d_i$ are taken randomly from a
uniform distribution in~$[0, 1]$, and $b=5$.  We compute the solution
to this problem, to use it as a benchmark, with the Optimization
Toolbox using the solver \textit{fmincon} with an \textit{interior
  point} algorithm.  Since the graph is connected, it follows that
$B=1$ in the definition of joint connectivity. Also, the constant of
nondegeneracy is~$\degn=0.25$ and $\sigmamax(\lap)\approx1.34$.  With
these values, we derive from~\eqref{eq:eta-cstep-condition} the
theoretically feasible consensus stepsize $\cstep=0.2475$.  For the
projection step in~\eqref{eq:minmax-dyn-constrained-agent-projection}
of the C-SP-SG algorithm, the bound on the optimal dual
set~\eqref{eq:compact-set-multipliers}, using the Slater vector
$\tilde{w}=\ones_N$ and $\zbar=0$, is
\begin{align*}
  r=\frac{N \max_{j\in\until{N}} c_j}{\log(2)\sum_{i=1}^N d_i -N/10} =
  3.313 .
\end{align*}
%
%
For comparison, we have also simulated the Consensus-Based Dual
Decomposition (CoBa-DD) algorithm proposed in~\cite{AS-HJR:15} using
(and adapting to this problem) the code made available online by the
authors\footnote{The Matlab code is available at
  \url{http://ens.ewi.tudelft.nl/~asimonetto/NumericalExample.zip}.}. (The
bound for the optimal dual set used in the projection of the estimates
of the multipliers is the same as above.)
%
We should note that the analysis in~\cite{AS-HJR:15} only considers
constant learning rates, which necessarily results in steady-state
error in the algorithm convergence.

We have simulated the C-SP-SG and the CoBa-DD algorithms in two
scenarios: under the Doubling Trick scheme of
Assumption~\ref{ass:Doubling-trick} (solid blue and magenta dash-dot
lines, respectively), and under constant learning rates equal
to~$0.05$ (darker grey) and~$0.2$ (lighter
grey). Fig.~\ref{fig:log_constraint} shows the saddle-point evaluation
error for both algorithms.  The saddle-point evaluation error of our
algorithm is well within the theoretical bound established in
Corollary~\ref{cor:app-constrained-opt}, which for this optimization
problem is approx. $1.18\times 10^9/\sqrt{t}$.  (This theoretical
bound is overly conservative for connected digraphs because the
ultimate bound for the disagreement~$\consissu$
in~\eqref{eq:def-constant-disagreement}, here
$\consissu\approx3.6\times 10^6$, is tailored for sequences of
digraphs that are $B$-jointly connected instead of relying on the
second smallest eigenvalue of the Laplacian of connected graphs.)
%
%
Fig.~\ref{fig:log_constraint_cost_and_constraint_evolution} compares
the network cost-error and the constraint satisfaction.
We can observe that the C-SP-SG and the CoBa-DD~\cite{AS-HJR:15}
algorithms have some characteristics in common:
\begin{itemize}
\item They both benefit from using the Doubling Trick scheme.
\item They approximate the solution, in all metrics of
  Fig.~\ref{fig:log_constraint} and
  Fig.~\ref{fig:log_constraint_cost_and_constraint_evolution} at a
  similar rate.  Although the factor in logarithmic scale of the
  C-SP-SG algorithm is larger, we note that this algorithm does not
  require the agents to solve a local optimization problem at each
  iteration for the updates of the primal variables, while both
  algorithms share the same communication complexity.
\item The empirical convergence rate for the saddle-point evaluation
error under the Doubling Trick scheme is of order~$1/\sqrt{t}$ (logarithmic slope~$-1/2$), while the
empirical convergence rate for the cost error under constant learning rates is of order~$1/t$
(logarithmic slope~$-1$). This is consistent with the theoretical
results here and in~\cite{AS-HJR:15} (wherein the theoretical bound
concerns the practical convergence of the cost error using constant
learning rates).
\end{itemize}
 
%
%


\begin{figure}[bth]
  \centering
  {\includegraphics[width=.9\linewidth]{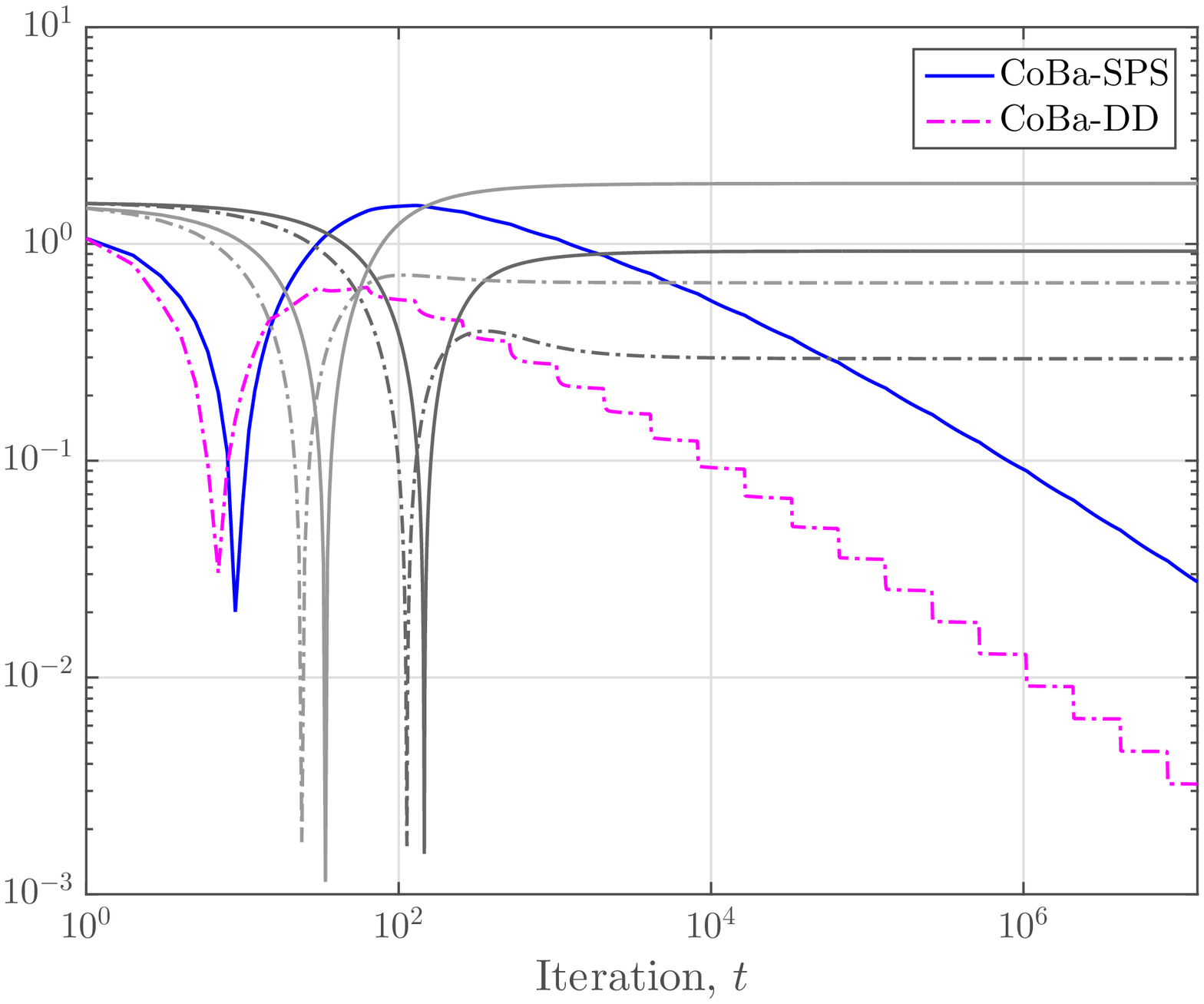}}
  \begin{picture}(0,0)(0,0)
  \end{picture}
  \caption{\color{black} Saddle-point evaluation error
    $\absolute{\phicc(\wntav,\zntav) -\lag(\wnstarwo, z^*)}$.
    The lines in grey represent the same algorithms simulated with
    constant learning rates equal to $0.2$ (lighter grey) and $0.05$
    (darker grey), respectively.}\label{fig:log_constraint}
\end{figure}

%
\begin{figure}[bth]
  \centering
  \subfigure[Cost
  error]{\includegraphics[width=.9\linewidth]{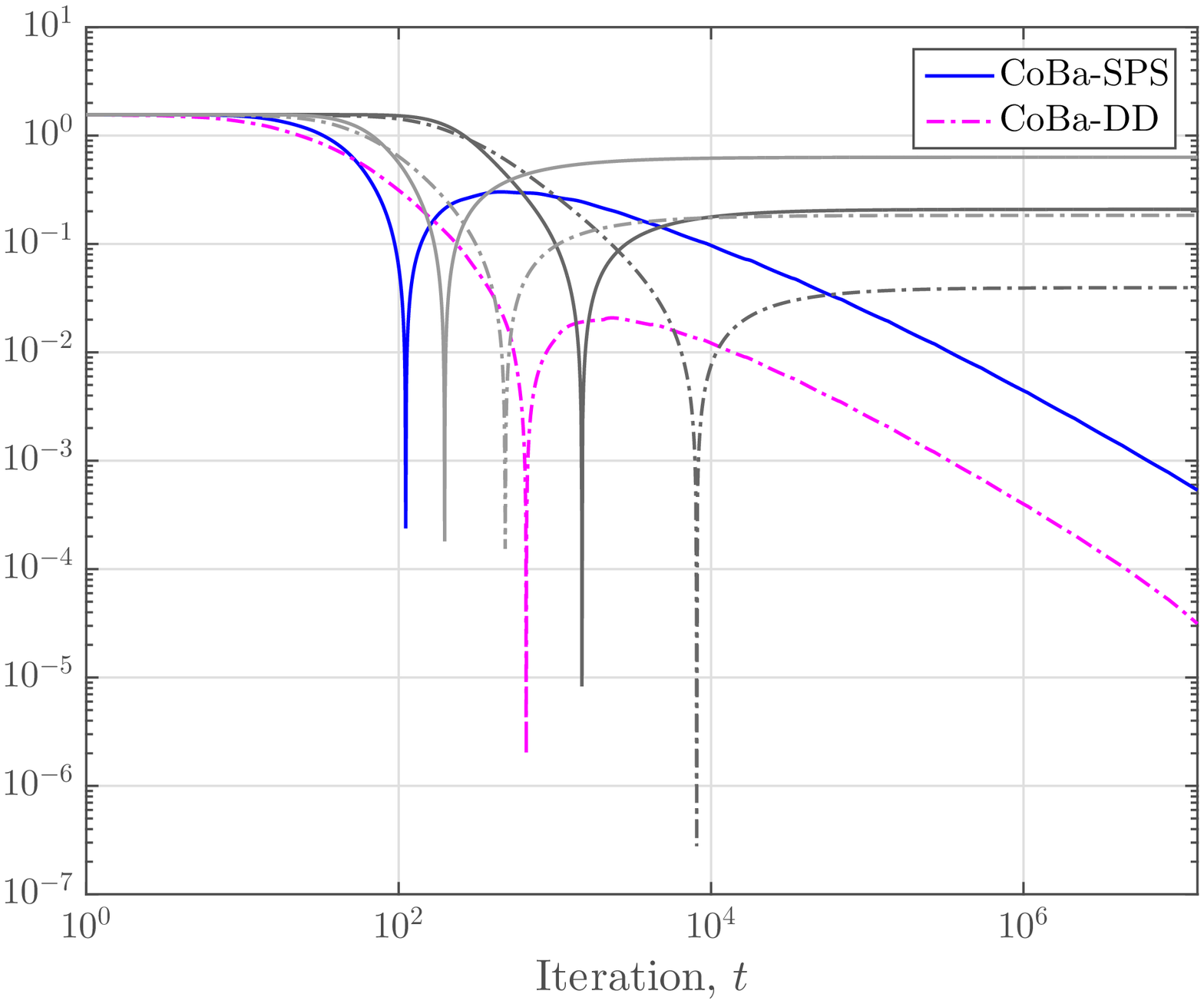}}
  \begin{picture}(0,0)(0,0)
  \end{picture}
  \subfigure[Constraint satisfaction]
  {\hspace{-4pt}\includegraphics[width=.9\linewidth]{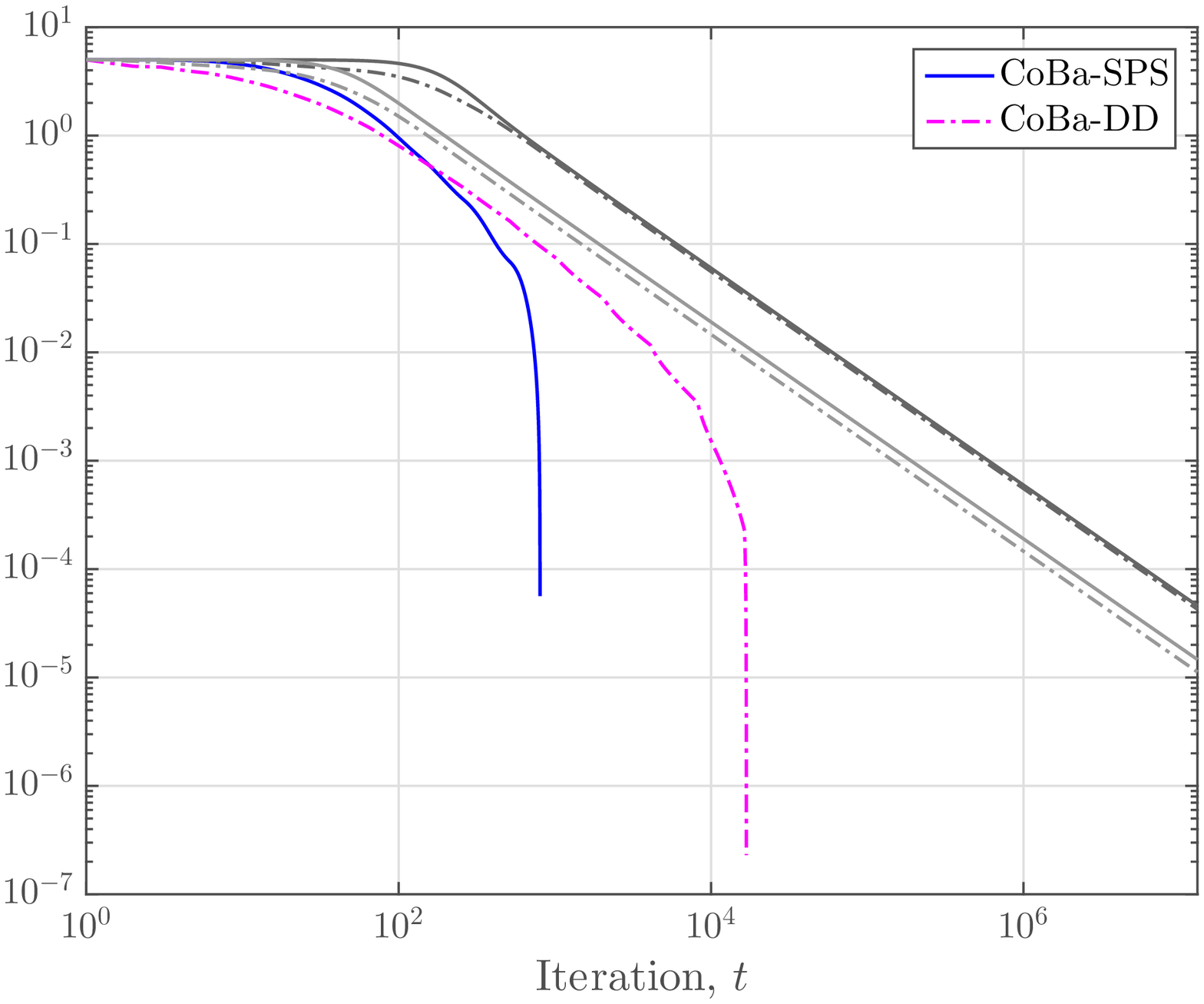}}
  \caption{\color{black} Cost error and constraint satisfaction.  For
    the same instantiations as in Fig.~\ref{fig:log_constraint}, (a)
    represents the evolution of the network cost
    error~$\absolute{\sum_{i=1}^N c_i \witav- \sum_{i=1}^N c_iw_i^*
    }$, and (b) the evolution of the network constraint satisfaction
    $-\sum_{i=1}^{N} d_i \log(1 + \witav) +b$.
    %
  }\label{fig:log_constraint_cost_and_constraint_evolution}
\end{figure}

}

\section{Conclusions and ideas for future work}

We have studied projected subgradient methods for saddle-point
problems under explicit agreement constraints.  We have shown that
separable constrained optimization problems can be written in this
form, where agreement plays a role in making distributed both the
objective function (via agreement on a subset of the primal variables)
and the constraints (via agreement on the dual variables). This
approach enables the use of existing consensus-based ideas to tackle
the algorithmic solution to these problems in a distributed fashion.
%
Future extensions will include, first, a refined analysis of
convergence for constrained optimization in terms of the cost
evaluation error instead of the saddle-point evaluation error.
Second, more general distributed algorithms for computing bounds on
Lagrange vectors and matrices, which are required in the design of
truncated projections preserving the optimal dual sets. (An
alternative route would explore the characterization of the intrinsic
boundedness properties of the proposed distributed dynamics.)
{\color{black} Third, the selection of other learning rates that
  improve the convergence rate of our proposed algorithms.}  Finally,
we envision applications to semidefinite programming, where chordal
sparsity allows to tackle problems that have the dimension of the
matrices grow with the network size, and also the treatment of
low-rank conditions.  {\color{black} Particular applications will
  include efficient optimization in wireless networks, control of
  camera networks, and estimation and control in smart grids.  }
 
%

\section*{Acknowledgments}
The authors thank the anonymous reviewers for their useful feedback
that helped us improve the presentation of the paper. This research
was partially supported by NSF Award CMMI-1300272.


\begin{thebibliography}{10}
	
	\bibitem{DMN-JC:15-cdc}
	D.~Mateos-N{\'u}{\~n}ez and J.~Cort{\'e}s, ``Distributed subgradient methods
	for saddle-point problems,'' in {\em {IEEE} Conf.\ on Decision and Control},
	(Osaka, Japan), pp.~5462--5467, 2015.
	
	\bibitem{KA-LH-HU:58}
	K.~Arrow, L.~Hurwitz, and H.~Uzawa, {\em Studies in Linear and Non-Linear
		Programming}.
	\newblock Stanford, California: Stanford University Press, 1958.
	
	\bibitem{AN-AO:09-jota}
	A.~Nedi{\'c} and A.~Ozdaglar, ``Subgradient methods for saddle-point
	problems,'' {\em Journal of Optimization Theory \& Applications}, vol.~142,
	no.~1, pp.~205--228, 2009.
	
	\bibitem{NP-SB:13}
	N.~Parikh and S.~Boyd, ``Proximal algorithms,'' vol.~1, no.~3, pp.~123--231,
	2013.
	
	\bibitem{SB-NP-EC-BP-JE:11}
	S.~Boyd, N.~Parikh, E.~Chu, B.~Peleato, and J.~Eckstein, ``Distributed
	optimization and statistical learning via the alternating direction method of
	multipliers,'' {\em Foundations and Trends in Machine Learning}, vol.~3,
	no.~1, pp.~1--122, 2011.
	
	\bibitem{BJ-TK-MJ-KHJ:08}
	B.~Johansson, T.~Keviczky, M.~Johansson, and K.~H. Johansson, ``Subgradient
	methods and consensus algorithms for solving convex optimization problems,''
	in {\em {IEEE} Conf.\ on Decision and Control}, (Cancun, Mexico),
	pp.~4185--4190, 2008.
	
	\bibitem{AN-AO:09}
	A.~Nedic and A.~Ozdaglar, ``Distributed subgradient methods for multi-agent
	optimization,'' {\em IEEE Transactions on Automatic Control}, vol.~54, no.~1,
	pp.~48--61, 2009.
	
	\bibitem{JW-NE:11}
	J.~Wang and N.~Elia, ``A control perspective for centralized and distributed
	convex optimization,'' in {\em {IEEE} Conf.\ on Decision and Control},
	(Orlando, Florida), pp.~3800--3805, 2011.
	
	\bibitem{MZ-SM:12}
	M.~Zhu and S.~Mart{\'\i}nez, ``On distributed convex optimization under
	inequality and equality constraints,'' {\em IEEE Transactions on Automatic
		Control}, vol.~57, no.~1, pp.~151--164, 2012.
	
	\bibitem{RC-GN-LS-DV:15}
	R.~Carli, G.~Notarstefano, L.~Schenato, and D.~Varagnolo, ``Distributed
	quadratic programming under asynchronous and lossy communications via
	{N}ewton-{R}aphson consensus,'' in {\em {E}uropean {C}ontrol {C}onference},
	(Lind, Austria), pp.~2514--2520, 2015.
	
	\bibitem{BG-JC:14-tac}
	B.~Gharesifard and J.~Cort{\'e}s, ``Distributed continuous-time convex
	optimization on weight-balanced digraphs,'' {\em IEEE Transactions on
		Automatic Control}, vol.~59, no.~3, pp.~781--786, 2014.
	
	\bibitem{JNT:84}
	J.~N. Tsitsiklis, {\em Problems in Decentralized Decision Making and
		Computation}.
	\newblock PhD thesis, Massachusetts Institute of Technology, Nov. 1984.
	\newblock Available at \url{http://web.mit.edu/jnt/www/Papers/PhD-84-jnt.pdf}.
	
	\bibitem{JNT-DPB-MA:86}
	J.~N. Tsitsiklis, D.~P. Bertsekas, and M.~Athans, ``Distributed asynchronous
	deterministic and stochastic gradient optimization algorithms,'' {\em IEEE
		Transactions on Automatic Control}, vol.~31, no.~9, pp.~803--812, 1986.
	
	\bibitem{DPB-JNT:97}
	D.~P. Bertsekas and J.~N. Tsitsiklis, {\em Parallel and Distributed
		Computation: Numerical Methods}.
	\newblock Athena Scientific, 1997.
	
	\bibitem{DPB-AN-AEO:03}
	D.~P. Bertsekas, A.~Nedi{\'c}, and A.~E. Ozdaglar, {\em Convex Analysis and
		Optimization}.
	\newblock Belmont, MA: Athena Scientific, 1st~ed., 2003.
	
	\bibitem{GS-DPP-FF-JSP:10}
	G.~Scutari, D.~P. Palomar, F.~Facchinei, and J.~S. Pang, ``Convex optimization,
	game theory, and variational inequality theory,'' {\em IEEE Signal Processing
		Magazine}, vol.~27, no.~3, pp.~35--49, 2010.
	
	\bibitem{SD:80}
	S.~Dafermos, ``Traffic equilibrium and variational inequalities,'' {\em
		Transportation Science}, vol.~14, no.~1, pp.~42--54, 1980.
	
	\bibitem{AN-MY-AHM-LSN:10}
	A.~Nagurney, M.~Yu, A.~H. Masoumi, and L.~S. Nagurney, {\em Networks Against
		Time: Supply Chain Analytics for Perishable Products}.
	\newblock SpringerBriefs in Optimization, Springer, 2010.
	
	\bibitem{AN-AO:10-siam}
	A.~Nedi{\'c} and A.~Ozdaglar, ``Approximate primal solutions and rate analysis
	for dual subgradient methods,'' {\em SIAM J. Optimization}, vol.~19, no.~4,
	pp.~1757--1780, 2010.
	
	\bibitem{JBHU-CL:93}
	J.-B. Hiriart-Urruty and C.~Lemar\'echal, {\em Convex Analysis and Minimization
		Algorithms I}.
	\newblock Grundlehren Text Editions, New York: Springer, 1993.
	
	\bibitem{MB-GN-FA:14}
	M.~B\"urger, G.~Notarstefano, and F.~Allg\"ower, ``A polyhedral approximation
	framework for convex and robust distributed optimization,'' {\em IEEE
		Transactions on Automatic Control}, vol.~59, no.~2, pp.~384--395, 2014.
	
	\bibitem{THC-AN-AS:14}
	T.-H. Chang, A.~Nedi{\'c}, and A.~Scaglione, ``Distributed constrained
	optimization by consensus-based primal-dual perturbation method,'' {\em IEEE
		Transactions on Automatic Control}, vol.~59, no.~6, pp.~1524--1538, 2014.
	
	\bibitem{AS-HJR:15}
	A.~Simonetto and H.~Jamali-Rad, ``Primal recovery from consensus-based dual
	decomposition for distributed convex optimization,'' {\em Journal of
		Optimization Theory and Applications}, vol.~168, no.~1, pp.~172--197, 2016.
	
	\bibitem{DY-SX-HZ:11}
	D.~Yuan, S.~Xu, and H.~Zhao, ``Distributed primal-dual subgradient method for
	multiagent optimization via consensus algorithms,'' {\em IEEE Trans. Systems,
		Man, and Cybernetics- Part B}, vol.~41, no.~6, pp.~1715--1724, 2011.
	
	\bibitem{DY-DWCH-SX:15}
	D.~Yuan, D.~W.~C. Ho, and S.~Xu, ``Regularized primal-dual subgradient method
	for distributed constrained optimization,'' {\em IEEE Transactions on
		Cybernetics}, 2015.
	\newblock To appear.
	
	\bibitem{AN-AO-PAP:10}
	A.~Nedic, A.~Ozdaglar, and P.~A. Parrilo, ``Constrained consensus and
	optimization in multi-agent networks,'' {\em IEEE Transactions on Automatic
		Control}, vol.~55, no.~4, pp.~922--938, 2010.
	
	\bibitem{IN-ID-JAKS:10}
	I.~Necoara, I.~Dumitrache, and J.~A.~K. Suykens, ``Fast primal-dual projected
	linear iterations for distributed consensus in constrained convex
	optimization,'' in {\em {IEEE} Conf.\ on Decision and Control}, (Atlanta,
	GA), pp.~1366--1371, Dec. 2010.
	
	\bibitem{DMA-TR-DS:10}
	D.~Mosk-Aoyama, T.~Roughgarden, and D.~Shah, ``Fully distributed algorithms for
	convex optimization problems,'' {\em SIAM J. Optimization}, vol.~20, no.~6,
	pp.~3260--3279, 2010.
	
	\bibitem{GN-FB:11}
	G.~Notarstefano and F.~Bullo, ``Distributed abstract optimization via
	constraints consensus: Theory and applications,'' {\em IEEE Transactions on
		Automatic Control}, vol.~56, no.~10, pp.~2247--2261, 2011.
	
	\bibitem{DR-JC:15-tac}
	D.~Richert and J.~Cort{\'e}s, ``Robust distributed linear programming,'' {\em
		IEEE Transactions on Automatic Control}, vol.~60, no.~10, pp.~2567--2582,
	2015.
	
	\bibitem{AAM-CD-AKRC-JAF:14}
	A.~A. Morye, C.~Ding, A.~K. Roy-Chowdhury, and J.~A. Farrell, ``Distributed
	constrained optimization for {B}ayesian opportunistic visual sensing,'' {\em
		IEEE Transactions on Control Systems Technology}, vol.~22, no.~6,
	pp.~2302--2318, 2014.
	
	\bibitem{AK-JL:15}
	A.~Kalbat and J.~Lavaei, ``A fast distributed algorithm for decomposable
	semidefinite programs,'' 2015.
	\newblock Available electronically at
	\url{http://www.ee.columbia.edu/~lavaei/ADMM_SDP_2015.pdf}.
	
	\bibitem{DMN-JC:15-necsys}
	D.~Mateos-N{\'u}{\~n}ez and J.~Cort{\'e}s, ``Distributed optimization for
	multi-task learning via nuclear-norm approximation,'' in {\em IFAC Workshop
		on Distributed Estimation and Control in Networked Systems}, vol.~48,
	(Philadelphia, PA), pp.~64--69, 2015.
	
	\bibitem{FB-JC-SM:08cor}
	F.~Bullo, J.~Cort{\'e}s, and S.~Mart{\'\i}nez, {\em Distributed Control of
		Robotic Networks}.
	\newblock Applied Mathematics Series, Princeton University Press, 2009.
	\newblock Electronically available at \url{http://coordinationbook.info}.
	
	\bibitem{SB-LV:09}
	S.~Boyd and L.~Vandenberghe, {\em Convex Optimization}.
	\newblock Cambridge University Press, 2009.
	
	\bibitem{DPB:99}
	D.~P. Bertsekas, {\em Nonlinear Programming}.
	\newblock Belmont, MA: Athena Scientific, 2nd~ed., 1999.
	
	\bibitem{AN-AO:10}
	A.~Nedi{\'c} and A.~Ozdaglar, ``Cooperative distributed multi-agent
	optimization,'' in {\em Convex Optimization in Signal Processing and
		Communications} (Y.~Eldar and D.~Palomar, eds.), Cambridge University Press,
	2010.
	
	\bibitem{RTR-RJBW:98}
	R.~T. Rockafellar and R.~J.~B. Wets, {\em Variational Analysis}, vol.~317 of
	{\em Comprehensive Studies in Mathematics}.
	\newblock New York: Springer, 1998.
	
	\bibitem{SSS:12}
	S.~Shalev-Shwartz, {\em Online Learning and Online Convex Optimization},
	vol.~12 of {\em Foundations and Trends in Machine Learning}.
	\newblock Now Publishers Inc, 2012.
	
	\bibitem{DMN-JC:14-tnse}
	D.~Mateos-N{\'u}{\~n}ez and J.~Cort{\'e}s, ``Distributed online convex
	optimization over jointly connected digraphs,'' {\em {IEEE} Transactions on
		Network Science and Engineering}, vol.~1, no.~1, pp.~23--37, 2014.
	
	\bibitem{DJW-SHS:98}
	D.~J. Watts and S.~H. Strogatz, ``Collective dynamics of `small-world'
	networks,'' {\em Nature}, vol.~393, pp.~440--442, 1998.
	
	\bibitem{RAH-CRJ:85}
	R.~A. Horn and C.~R. Johnson, {\em Matrix Analysis}.
	\newblock Cambridge University Press, 1985.
	
\end{thebibliography}

\appendix

\newcounter{mycounter}
\renewcommand{\themycounter}{A.\arabic{mycounter}}
\newtheorem{definitionappendix}[mycounter]{Definition}
\newtheorem{lemmaappendix}[mycounter]{Lemma}
\newtheorem{theoremappendix}[mycounter]{Theorem}

\section{Appendix}\label{app}

{\color{black} Here we present the proofs of the results
  Lemma~\ref{le:basic-bound-convex-concave-function-differences} and
  Proposition~\ref{prop:cum-disagreement-projected-subg-saddle} stated
  in Section~\ref{sec:convergence-analysis}.  }

\begin{proof}[Proof of Lemma~\ref{le:basic-bound-convex-concave-function-differences}]
  In this proof we extend the saddle-point analysis for the
  (centralized) subgradient methods in~\cite[Lemma 3.1]{AN-AO:09-jota}
  by incorporating the treatment on the disagreement from our previous
  work in~\cite[Lemma V.2]{DMN-JC:14-tnse}.  We first define
  \begin{align}\label{eq:def-residual-projection-wn-xn}
    \begin{bmatrix}
      \rnwtp
      \\
      \rnxtp
    \end{bmatrix}
    \defin
    \begin{bmatrix}
      \wntp-\wntphat
      \\
      \dnmtp-\dnmtphat
    \end{bmatrix}.
  \end{align}
  Since $\identity_d-\step\tlap$ is a stochastic matrix (because
  $\cstep$~satisfies~\eqref{eq:upper-cons-cstep}),
  then its product by any vector is a convex combination of the
  entries of the vector. Hence, the fact that $\dnmt\in\dnmset$
  implies that $\dnmt-\cstep\tlapbk\dnmt\in\dnmset$.
  {\color{black}Using this together with the definition of orthogonal
    projection~\eqref{eq:def-projection}, we get}
  \begin{align}\label{eq:bound-residual-projec}
    \norm{\rnxtp}=&\, \norm{\projec{\dnmset}{\dnmtphat}-\dnmtphat}
    \notag
    \\
    \le&\,\norm{(\dnmt-\cstep\tlapbk\dnmt)-\dnmtphat}
    =\gradstept\norm{\subgdnmt}.
  \end{align}

  Similarly, since $\wnt\in\wnsetan$, we also have
  \begin{align*} 
    \norm{\rnwtp}=&\, \norm{\projec{\wnsetan}{\wntphat}-\wntphat}
    \\
    \le&\,\norm{\wnt-\wntphat} =\gradstept\norm{\subgwnt}.
  \end{align*}
  Left-multiplying the dynamics of $\wnt$ and $\dnmt$
  from~\eqref{eq:minmax-dyn-wn} and~\eqref{eq:minmax-dyn-xn} (in terms
  of the residual~\eqref{eq:def-residual-projection-wn-xn}) by the
  block-diagonal matrix $\diag(\identity_{Nd},\mbk)$, and using
  $\mbk\tlapbk=0$, we obtain
  \begin{align}\label{eq:dynamics-projected-agreement-agents-decisions}
    \begin{bmatrix}
      \wntp
      \\
      \mbk\dnmtp
    \end{bmatrix}
    =
    \begin{bmatrix}
      \wnt
      \\
      \mbk\dnmt
    \end{bmatrix}+
    \begin{bmatrix}
      -\gradstept\subgwnt+\rnwtp
      \\
      -\gradstept\mbk\subgdnmt+\mbk\rnxtp
    \end{bmatrix}.
  \end{align}
  Subtracting $(\wnp, \dnmp)\in\wnsetan\times\dnmset$
  on each side, taking the norm, and noting that $\mbk\tp=\mbk$ and
  $\mbk^2=\mbk$,
  we get
  \begin{align}\label{eq:difference-square-norm-first-step}
    &\norm{\wntp-\wnp}^2+ \norm{\mbk\dnmtp-\dnmp}^2
    \\
    = &\,\norm{\wnt-\wnp}^2+\norm{\mbk\dnmt-\dnmp}^2 \notag
    \\
    &\,+\norm{-\gradstept \subgwnt+\rnwtp}^2 +\norm{-\gradstept
      \mbk\subgdnmt+\mbk\rnxtp}^2 \notag
    \\
    &\!-2\gradstept\subgwnt\tp(\wnt-\wnp)
    -2\gradstept\subgdnmt\tp(\mbk\dnmt-\mbk\dnmp) \notag
    \\
    &\, +2\rnwtp\tp(\wnt-\wnp) +2\rnxtp\tp(\mbk\dnmt-\mbk\dnmp).
    \notag
  \end{align}
  We can bound the term $-\subgdnmt\tp(\mbk\dnmt-\mbk\dnmp)$ by
  subtracting and adding $\dnmt-\dnmp$ inside the bracket and using
  convexity,
  \begin{align}\label{eq:convexity-wrt-dnmt}
    &\,-\subgwnt\tp(\wnt-\wnp) -\subgdnmt\tp(\mbk\dnmt-\mbk\dnmp)
    \\
    =&\, -\subgdnmt\tp(\mbk\dnmt-\dnmt) -\subgdnmt\tp(\dnmp-\mbk\dnmp)
    \nonumber
    \\
    &\, -
    \begin{bmatrix}
      \subgwnt\tp
      &
      \subgdnmt\tp
    \end{bmatrix}
    \begin{bmatrix}
      \wnt-\wnp
      \\
      \dnmt-\dnmp
    \end{bmatrix}
    \nonumber
    \\
    \le &\; \subgdnmt\tp \lapkbk\dnmt-\subgdnmt\tp \lapkbk\dnmp
    \nonumber
    \\
    &\,+\phicc(\wnp,\dnmp,\munt,\znt)-\phicc(\wnt,\dnmt,\munt,\znt)\,,
    \nonumber
  \end{align}
  where we have used $\lapkbk = \identity_{N\dimdn}-\mbk$ and the
  fact that $\subgwnt\in\partial_{\wn}\phicc(\wnt,\dnmt,\munt,\znt)$
  and
  $\subgdnmt\in\partial_{\dnm}\phicc(\wnt,\dnmt,\munt,\znt)$. Using
  this bound and \eqref{eq:difference-square-norm-first-step}, we
  get
  \begin{align}\label{eq:bound-lagdnmt-lagdnm}
    &\,
    2(\phicc(\wnt,\dnmt,\munt,\znt)-\phicc(\wnp,\dnmp,\munt,\znt))
    \\
    \le&\, \tfrac{1}{\gradstept} \big(\norm{\wnt-\wnp}^2
    -\norm{\wntp-\wnp}^2\big) \notag
    \\
    &\, +\tfrac{1}{\gradstept} \big(\norm{\mbk\dnmt-\dnmp}^2
    -\norm{\mbk\dnmtp-\dnmp}^2\big) \notag
    \\
    &\,+2\subgdnmt\tp \lapkbk\dnmt-2\subgdnmt\tp \lapkbk\dnmp
    \notag
    \\
    &\,+\tfrac{1}{\gradstept}\norm{\!-\!\gradstept \subgwnt+\rnwtp}^2
    +  \tfrac{1}{\gradstept}\norm{\!-\!\gradstept \mbk\subgdnmt+\mbk\rnxtp}^2
    \notag
    \\
    &\,+\tfrac{2}{\gradstept}\rnwtp\tp(\wnt-\wnp)
    +\tfrac{2}{\gradstept}\rnxtp\tp(\mbk\dnmt-\dnmp).
    \notag
  \end{align}
  We now bound each of the terms in the last three lines
  of~\eqref{eq:bound-lagdnmt-lagdnm}. First, using the Cauchy-Schwarz
  inequality, we get
  \begin{align}\label{eq:bound-scalar-product-subg-and-disagreement}
    \subgdnmt\tp \lapkbk\dnmt-\subgdnmt\tp \lapkbk\dnmp
    \le\norm{\subgdnmt}(\norm{\lapkbk\dnmt}+\norm{\lapkbk\dnmp}).
  \end{align}
  For the terms in the second to last line, using the triangular
  inequality, the submultiplicativity of the norm, the fact that
  $\norm{\mbk}\le 1$, and the bound~\eqref{eq:bound-residual-projec},
  we have
  \begin{align}\label{eq:bound-subg-and-rnctp}
    &\,\!\!\!\norm{\!-\!\gradstept \mbk\subgdnmt+\mbk\rnxtp}
    \le\norm{\!-\!\gradstept \mbk\subgdnmt}+\norm{\mbk\rnxtp}
    \nonumber
    \\
    \le&\,\gradstept\norm{\mbk}\norm{\subgdnmt}+\norm{\mbk}\norm{\rnxtp}
    \le 2\gradstept \norm{\subgdnmt},
  \end{align}
  and, similarly, 
  \begin{align*} 
    &\,\norm{\!-\gradstept\subgwnt+\rnwtp}
    \le 2\gradstept \norm{\subgwnt}.
  \end{align*}
  Finally, regarding the term $\rnxtp\tp(\mbk\dnmt-\dnmp)$, we use the
  definition of $\rnxtp$ and also add and subtract $\dnmtphat$ inside
  the bracket. With the fact that $\mbk\dnmp\in\dnmset$ (because
  $\dmset$ is convex), we
  {\color{black} leverage the
    property~\eqref{eq:prelims-projection-property} of the orthogonal
    projection to derive the first inequality.  For the next two
    inequalities we use the Cauchy-Schwarz inequality, and then the
    bound in~\eqref{eq:bound-residual-projec} for the residual, and
    also the definition of~$\dnmtphat$, the fact that
    $\mbk\dnmt-\dnmt=-\lapkbk\dnmt$, and the triangular inequality.}
  Formally,
  \begin{align}\label{eq:bound-rnxtp-dnmt-dnmp}
    &\,\rnxtp\tp(\mbk\dnmt-\mbk\dnmp)
    =\rnxtp\tp(\mbk\dnmt-\dnmtphat)
    \\
    \nonumber
    &\,+\big(\projec{\dnmset}{\dnmtphat}-\dnmtphat\big)(\dnmtphat-\mbk\dnmp)
    \\
    \nonumber
    \le&\,\rnxtp\tp(\mbk\dnmt-\dnmtphat)
    \le\norm{\rnxtp}\norm{\mbk\dnmt-\dnmtphat}
    \\
    \le&\,\gradstept\norm{\subgdnmt}\,
    \norm{\!-\lapkbk\dnmt+\cstep\tlapb\dnmt+\gradstept\subgdnmt}
    \nonumber
    \\
    \le&\,\gradstept\norm{\subgdnmt}\Big((1+\cstep\lambdaup)\norm{\lapkbk\dnmt}+
    \gradstept\norm{\subgdnmt} \Big), \nonumber
  \end{align}
  where in the last inequality we have also used {\color{black} a bound
    for the term~$\norm{\tlapb\dnmt}$ invoking~$\lambdaup$ that we
    explain next. From the Courant-Fischer min-max
    Theorem~\cite{RAH-CRJ:85} applied to the matrices $\tlap\tp\tlap$
    and $\lapk^2$ (which are symmetric with the same nullspace), we
    deduce that for any $x\in\real^N$,
    \begin{align*}
      \frac{x\tp\tlap\tp\tlap x
      }{\lambdamax(\tlap\tp\tlap)} \le \frac{x\tp\lapk^2
        x}{\lambda_{n-1}(\lapk^2)} ,
    \end{align*}
    where $\lambda_{n-1}(\cdot)$ refers to the second smallest
    eigenvalue, which for the matrix $\lapk^2=\lapk$ is $1$. (Note
    that all its eigenvalues are $1$, except the smallest that is
    $0$.) With the analogous inequality for Kronecker products with
    the identity $\identity_d$, the bound needed to
    conclude~\eqref{eq:bound-rnxtp-dnmt-dnmp} is then
    \begin{align*} 
      \norm{\tlapb\dnmt}=&\, \sqrt{\dnmt\tp\tlapb\tp\tlapb\dnmt}
      \\
      \nonumber \le &\, \sqrt{\lambdamax(\tlapb\tp\tlapb)\:\dnmt\tp\lapkb^2\dnmt} =
      \sigmamax(\tlap)\:\norm{\lapkb\dnmt} .
    \end{align*}
  }
  %
  
  Similarly to~\eqref{eq:bound-rnxtp-dnmt-dnmp}, now without the
  disagreement terms,
  \begin{align*} 
    &\,\rnwtp\tp(\wnt-\wnp)
    \\
    =&\,\rnwtp\tp(\wnt-\wntphat)
    \,+\big(\projec{\wnsetan}{\wntphat}-\wntphat\big)(\wntphat-\wnp)
    \\
    \le&\,\rnwtp\tp(\wnt-\wntphat)
    \\
    \le&\, \norm{\rnwtp}\norm{\wnt-\wntphat}\le \gradstept^2\norm{\subgwnt}^2.
    \nonumber
  \end{align*}
  Substituting the
  bounds~\eqref{eq:bound-scalar-product-subg-and-disagreement},~\eqref{eq:bound-subg-and-rnctp}
  and~\eqref{eq:bound-rnxtp-dnmt-dnmp}, and their counterparts for
  $\wnt$, in~\eqref{eq:bound-lagdnmt-lagdnm}, we obtain
  \begin{align}\label{eq:bound-lagdnmt-lagdnm-substituted}
    &\,
    2(\phicc(\wnt,\dnmt,\munt,\znt)-\phicc(\wnp,\dnmp,\munt,\znt))
    \\
    \le&\, \tfrac{1}{\gradstept} \big(\norm{\wnt-\wnp}^2
    -\norm{\wntp-\wnp}^2\big) \notag
    \\
    &\,+ \tfrac{1}{\gradstept} \big(\norm{\mbk\dnmt-\dnmp}^2
    -\norm{\mbk\dnmtp-\dnmp}^2\big) \notag
    \\
    &\,+2\norm{\subgdnmt}(\norm{\lapkbk\dnmt}+\norm{\lapkbk\dnmp})
    \nonumber +6 \gradstept \norm{\subgwnt}^2 \notag
    \\
    &\,+ 4\gradstept \norm{\subgdnmt}^2
    +2\norm{\subgdnmt}\Big((1+\cstep\lambdaup)\norm{\lapkbk\dnmt}+
    \gradstept\norm{\subgdnmt} \Big) \notag
  \end{align}
  and~\eqref{eq:prop-difference-dnmt-xn} follows.
  The bound~\eqref{eq:prop-difference-znt-zn} can be derived
  similarly, requiring concavity of~$\phicc$ in~$(\mun,\zn)$.
\end{proof}

\begin{proof}[Proof of Proposition~\ref{prop:cum-disagreement-projected-subg-saddle}]
  Since both dynamics in~\eqref{eq:disagreement-minmax-dyn-wnxnzn} are
  structurally similar, we study the first one,
  \begin{align}\label{eq:consensus-projection-perturn-dnm}
    \dnmtp=\dnmt-\cstep\tlapb\dnmt+\pertdnmt +\rnxtp ,
  \end{align}
  where $\rnxtp$ is as in~\eqref{eq:def-residual-projection-wn-xn} and
  satisfies (similarly to \eqref{eq:bound-residual-projec}) that
  \begin{align*} 
    \norm{\rnxtp}=&\, \norm{\projec{\dnmset}{\dnmtphat}-\dnmtphat}
    \notag
    \\
    \le&\,\norm{(\dnmt-\cstep\tlapb\dnmt)-\dnmtphat}
    =\norm{\pertdnmt}.
  \end{align*}
  The dynamics~\eqref{eq:consensus-projection-perturn-dnm} coincides
  with that of~\cite[eqn.~(29)]{DMN-JC:14-tnse} where, in the notation
  of the reference, one sets $\inputet\defin\pertdnmt +\rnxtp$.
  Therefore, we obtain a bound analogous
  to~\cite[eqn.~(34)]{DMN-JC:14-tnse},
  \begin{align}\label{eq:norm-lapkb-rhos-dnm}
    &\,\norm{\lapkb\dnmt} \le
    \rho_\degnt^{\lceil\tfrac{t-1}{B}\rceil-2} \norm{\dnm_1}
    +\sum_{s=1}^{t-1} \rho_\degnt^{\lceil\tfrac{t-1-s}{B}\rceil-2}
    \norm{\inputes}\,,
  \end{align}
  where  $\rho_\degnt\defin 1-\frac{\degnt}{4 N^2}$ .
  To derive~\eqref{eq:iss-dnm} we use three facts: first
  $\norm{\inputet}\le\norm{\pertdnmt} +\norm{\rnxtp}\le
  2\norm{\pertdnmt}$; second, $\sum_{k=0}^{\infty}r^{k} =
  \tfrac{1}{1-r}$ for any $r \in (0,1)$ and in particular for
  $r=\rho_\degnt^{1/B}$; and third,
  \begin{align*}
    \rho_\degnt^{-1}=\tfrac{1}{1-\degnt/(4 N^2)}\le
    \tfrac{1}{1-1/(4N^2)} =\tfrac{4N^2}{4 N^2-1} \le \tfrac{4}{3}.
  \end{align*}
  {\color{black} The constant~$\consissu$ in the statement is obtained
    recalling that
    \begin{align*}
      r=\rho_\degnt^{1/B}=\Big(1-\frac{\degnt}{4 N^2}\Big)^{1/B} .
    \end{align*}
  } To obtain~\eqref{eq:cum-iss-dnm}, we
  sum~\eqref{eq:norm-lapkb-rhos-dnm} over the time horizon~$t'$ and
  bound the double sum as follows: using $r=\rho_\degnt^{1/B}$ for
  brevity, we have
  \begin{align*}
    \sum_{t=2}^{t'} \sum_{s=1}^{t-1} r ^{t-1-s} \norm{\inputes} & =
    \sum_{s=1}^{t'-1} \sum_{t=s+1}^{t'} r^{t-1-s} \norm{\inputes}
    \\
    \sum_{s=1}^{t'-1} \norm{\inputes}\sum_{t=s+1}^{t'} r^{t-1-s} & \le
    \frac{1}{1-r} \sum_{s=1}^{t'-1} \norm{\inputes} .
  \end{align*}
  Finally, we use again the bound $\norm{\inputet}\le
  2\norm{\pertdnmt}$. 
\end{proof}

\end{document}